\documentclass[11pt,a4paper]{amsart}
\usepackage{amsmath}
\usepackage{amssymb}
\usepackage{graphicx}
\usepackage{color}
\usepackage{lipsum}

\usepackage{enumerate}

\theoremstyle{plain}
\newtheorem{Main}{Theorem}

\newtheorem{Thm}{Theorem}
\newtheorem{Prop}[Thm]{Proposition}
\newtheorem{Lem}[Thm]{Lemma}

\newtheorem{Rem}[Thm]{Remark}
\theoremstyle{remark}

\def\max{\operatorname{max}}
\def\min{\operatorname{min}}

\begin{document}

\title[]
{Dimension approximation for diffeomorphisms preserving hyperbolic SRB measures}

\author{Juan Wang}
\address{School of Mathematics, Physics and Statistics, Shanghai University of Engineering Science, Shanghai 201620, P.R. China}
\email{wangjuanmath@sues.edu.cn}

\author{Congcong Qu\textsuperscript{*}}
\address{Department of mathematics, Soochow University, Suzhou 215006, Jiangsu, P.R. China}
\email{congcongqu@foxmail.com}

\author{Yongluo Cao }
\address{Departament of Mathematics, Shanghai Key Laboratory of PMMP, East China Normal University,
 Shanghai 200062, P.R. China}
\address{Departament of Mathematics, Soochow University,
Suzhou 215006, Jiangsu, P.R. China}
\email{ylcao@suda.edu.cn}


\newcommand\blfootnote[1]{%
\begingroup
\renewcommand\thefootnote{}\footnote{#1}%
\addtocounter{footnote}{-1}%
\endgroup
}

\thanks{The first author is partially supported by NSFC (11501400, 11871361) and the Talent Program of Shanghai University of Engineering Science. The third author is partially supported by NSFC(11771317, 11790274) and Science and Technology Commission of Shanghai Municipality (No.18dz2271000).}

\subjclass[2010] {37C45, 37D35, 37D20
 }

\keywords{Hausdorff dimension, hyperbolic measures, hyperbolic sets}

\blfootnote{\textsuperscript{*}Corresponding author}

\begin{abstract}
For a $C^{1+\alpha}$ diffeomorphism $f$ preserving a hyperbolic ergodic SRB measure $\mu$, Katok's remarkable results assert that $\mu$ can be approximated by a sequence of hyperbolic sets $\{\Lambda_n\}_{n\geq1}$. In this paper, we prove the Hausdorff dimension for $\Lambda_n$ on the unstable manifold tends to the dimension of the
unstable manifold. Furthermore, if the stable direction is one dimension, then the Hausdorff dimension of $\mu$ can be approximated by the Hausdorff dimension of $\Lambda_n$.

To establish these results, we utilize the $u$-Gibbs property of the conditional measure of the equilibrium measure of $-\psi^{s}(\cdot,f^n)$  and the properties of the uniformly hyperbolic dynamical systems.
\end{abstract}

\maketitle


\section{Introduction}

The present paper is motivated by Cao, Pesin and Zhao \cite{caopesinzhao2019} and Climenhaga, Pesin and Zelerowicz \cite{CliPZ}.
Let $f:M\rightarrow M$ be a $C^{1+\alpha}$ diffeomorphism on a $d_{0}$-dimensional  compact Riemannian manifold and $\mu$ be a hyperbolic ergodic SRB $f$-invariant probability measure. $\mu$ is said to be an {\it SRB measure}, if it is absolutely continuous along the unstable leaves. Let $\Gamma$ be the set of points which are regular in the sense of Oseledec \cite{Oseledets}. For  every $x\in\Gamma$, denote its distinct Lyapunov exponents by
\begin{align*}
	\lambda_{1}(\mu)>...>\lambda_{\ell}(\mu)>0>\lambda_{\ell+1}(\mu)>...>\lambda_{k}(\mu)
\end{align*}
with multiplicities $m_{1}, m_{2},...,m_{k}\geq 1$	and let
\begin{align}\label{tangent}
	T_{x}M=E_{1}(x)\oplus E_{2}(x)\oplus...\oplus E_{k}(x)
\end{align}
be the corresponding decomposition of its tangent space, where $0<k\leq d_{0}$. Define
\begin{align*}
	W^{u}_{\text{loc}}(x, f) = \{y\in M:d(f^{-n}(x),f^{-n}(y))\leq \beta \text{~for all~} n\geq 0\},
\end{align*}
where $\beta>0$ is small. It is called the {\it local unstable manifold of $f$ at $x$ with respect to $\mu$}. This is an embedded $C^{1}$ disk with $T_{x}W^{u}_{\text{loc}}(x, f)=E_{1}(x)\oplus...\oplus E_{\ell}(x)$. {\it The unstable manifold of $f$ at $x$ with respect to $\mu$} is given as follows:
\begin{equation*}
	W^{u}(x, f) = \bigcup_{n\geq 0}f^{n}(W^{u}_{\text{loc}}(f^{-n}(x), f)).
\end{equation*}
Similarly we can define {\it the stable manifold of $f$ at $x$ with respect to $\mu$}:
\begin{equation*}
	W^{s}(x, f) = \bigcup_{n\geq 0}f^{-n}(W^{s}_{\text{loc}}(f^{n}(x), f)),
\end{equation*}
where
\begin{align*}
	W^{s}_{\text{loc}}(x, f) = \{y\in M:d(f^{n}(x),f^{n}(y))\leq \beta \text{~for all~} n\geq 0\}
\end{align*}
is the {\it local stable manifold of $f$ at $x$ with respect to $\mu$}.

For a $C^{2}$ diffeomorphism $f$ preserving an ergodic hyperbolic measure with positive entropy, Katok \cite{katok1980} proved there exists a sequence of horseshoes and the topological entropy of $f$ restricted to horseshoes can be arbitrarily close to the measure-theoretic entropy. Avila, Crovisier and Wilkinson \cite{ACW} extended this results and explicitly gave a dominated splitting $TM=E_{1}\oplus_{>}...\oplus_{>} E_{k}$ corresponding to Oseledec subspaces on horseshoes and the approximation of Lyapunov exponents on each sub-bundle $E_{j}$. They used this result to study the density of stable ergodicity. Mendoza  \cite{Mendoza1988} used Katok's results to obtain that the Hausdorff dimensions for horseshoes on the local unstable manifolds converge to one for a $C^{2}$ surface diffeomorphism preserving an ergodic hyperbolic SRB measure.
Cao, Pesin, Zhao \cite{caopesinzhao2019} considered an ergodic invariant measure $\mu$ with positive entropy for $C^{1+\alpha}$ non-conformal repellers,
and constructed a compact expanding invariant set with dominated splitting corresponding to Oseledec splitting of $\mu$, for which entropy and Lyapunov exponents approximate to entropy and Lyapunov exponents for $\mu$. Moreover, they used this construction to give a sharp estimate for the lower bound estimate of Hausdorff dimension of non-conformal repellers.
In this work we exploit Cao, Pesin and Zhao's ideas in an essential way to get the first main result, which generalizes Mendoza's result in \cite{Mendoza1988} for diffeomorphisms in higher dimensional manifold.
S$\acute{\text{a}}$nchez-Salas \cite{S} also proved similar result. His proof  is based on Markov towers that can also be described by horseshoes with infinitely many branches and variable return times. However, due to a serious flaw in the proof  for key proposition 5.1   in that paper, a complete proof for the statement in \cite{S} does not exist so far.   In this paper, utilizing the u-Gibbs property of the conditional measure of the equilibrium measure of $-\psi^{s}(\cdot,f^n)$ (see the definition in (\ref{star2})) and the properties of the uniformly hyperbolic dynamical systems, we can prove the following theorem. 



\begin{Main}\label{hyperbolic1}
Suppose $f: M\rightarrow M$ is a $C^{1+\alpha}$ diffeomorphism on a $d_0$-dimensional compact Riemannian manifold $M$ and $\mu$ is a hyperbolic ergodic SRB $f$-invariant
probability measure. Then there is a sequence of hyperbolic sets $\Lambda_{n}$ and a sequence of ergodic measures $\mu_{n}$ supported on $\Lambda_{n}$ such that $\mu_{n}\rightarrow \mu$ in the weak$^{\ast}$ topology and
\begin{align*}
\dim_{H}(\Lambda_{n}\cap W^{u}_{\text{loc}}(x, f))\rightarrow \dim W^{u}(x, f) (\text{as } n \to \infty)
\end{align*}
uniformly for $x\in \Lambda_{n}$.
\end{Main}

S$\acute{\text{a}}$nchez-Salas \cite{S2} considered the expanding measure for a $C^{2}$ interval map with finitely many non-degenerate critical points, and
proved that the Hausdorff dimension of the expanding measure can be approximated by the Hausdorff dimension of expanding Cantor sets.
His method is relying on the geometrical properties of the base of Markov towers, and Ledrappier-Young characterization of the expanding measure satisfying the Rokhlin-Pesin entropy
formula.

Barreira, Pesin and Schmeling \cite{BPS} showed that the pointwise dimension of a hyperbolic measure for $C^{1+\alpha}$ diffeomorphism is almost everywhere the sum of the pointwise dimensions along stable
and unstable local manifolds.
Our second motivation is to  apply Barreira, Pesin and Schmeling's result and Theorem \ref{hyperbolic1} to extend S$\acute{\text{a}}$nchez-Salas's result \cite{S2} to  hyperbolic SRB measures for diffeomorphism in higher dimensional. The investigation of dimension for hyperbolic set is fairly well understood for two dimensional systems, which  is essentially a conformal setting. However, extending the theory to higher-dimensional and genuinely nonconformal situations is well known to be difficult.  On the other hand, the stable/unstable holonomies for hyperbolic surface diffeomorphisms are always bi-Lipschitz. But in a higher-dimensional setting in general stable/unstable holonomies are not bi-Lipschitz but only H$\ddot{\text{o}}$lder continuous.


Let $f: M\to M$ be a $C^{1+\alpha}$ diffeomorphism of a compact Riemannian surface, and $\mu$ be a hyperbolic ergodic $f$-invariant probability measure. Young \cite{Y} proved
\begin{equation}\label{Young1}
\text{dim}_H\mu = \frac{h_\mu(f)}{\lambda_1(\mu)}-\frac{h_\mu(f)}{\lambda_2(\mu)},
\end{equation}
where $h_\mu(f)$ is the Kolmogorov-Sinai entropy of $f$ with respect to $\mu$, and $\lambda_1(\mu)$ and $\lambda_2(\mu)$ are respectively the positive and negative Lyapunov exponents.
A measurable partition $\xi^u$ ($\xi^s$) of $M$ is said to be {\it subordinate to the unstable (stable) manifold} if for $\mu$-a.e.~$x$, $\xi^u(x)\subset W^u(x,f)$ ($\xi^s(x)\subset W^s(x,f)$) and contains an open neighborhood of $x$ in $W^u(x,f)$ ($W^s(x,f)$). Let $\{\mu_x^u\}$  and $\{\mu_x^s\}$ be systems of conditional measures associated with $\xi^u$ and $\xi^s$ respectively.
For $x\in\Gamma$, define
$$d^u(x)=\lim_{r\to0}\frac{\log\mu_x^u(B^u(x, r))}{\log r} \text{ and } d^s(x)=\lim_{r\to0}\frac{\log\mu_x^s(B^s(x, r))}{\log r},$$
which are well defined (see \cite{LY}). Here $B^*(x, r)=\{y\in W^*(x,f): d_*(x,y)<r\}$ with $*\in\{u, s\}$, and $d_u$ ($d_s$) is the metric induced by the Riemannian structure on the unstable (stable) manifold $W^u$ ($W^s$).
Ledrappier and Young \cite{LY} proved for $\mu$ almost every $x\in M$, $d^*(x)$ is constant, and
we denote the constant by $\text{dim}_H^*\mu$, where $*\in\{u,s\}$. They also established
\begin{equation*}
\text{dim}_H^u\mu=\frac{h_\mu(f)}{\lambda_1(\mu)} \text{ and } \text{dim}_H^s\mu=-\frac{h_\mu(f)}{\lambda_2(\mu)}.
\end{equation*}
For each $\varepsilon>0$, Katok \cite{katok1980} proved there is a hyperbolic set $\Lambda_\varepsilon$ satisfying
\begin{equation}\label{Katok1}
|h_{\text{top}}(f|_{\Lambda_\varepsilon})-h_\mu(f)|<\varepsilon \text{ and } |\lambda_i(\mu)-\lambda_i(\nu)|<\varepsilon
\end{equation}
for each $\nu\in\mathcal{M}_{inv}(f|_{\Lambda_\varepsilon})$ and $i=1,2$, where $\mathcal{M}_{inv}(f|_{\Lambda_\varepsilon})$ denotes the set of all the $f$-invariant Borel probability measures on $\Lambda_\varepsilon$.
For every $x\in\Lambda_\varepsilon$, in \cite{MM1983} McCluskey and Manning proved that
\begin{equation}\label{usdimension}
\text{dim}_H(\Lambda_\varepsilon\cap W^u(x, f))=t^u \text{ and } \text{dim}_H(\Lambda_\varepsilon\cap W^s(x, f))=t^s
\end{equation}
where $t^u$ and $t^s$ are the roots of
$P(f|_{\Lambda_\varepsilon}, -t\log\|Df|_{E^u}\|)=0$, $P(f|_{\Lambda_\varepsilon}, t\log\|Df|_{E^s}\|)=0$
respectively (here $P(\cdot)$ denotes the topological pressure). Combining with (\ref{Katok1}), one has
$\text{dim}_H(\Lambda_\varepsilon\cap W^u(x, f))\to \frac{h_\mu(f)}{\lambda_1(\mu)}$ and $\text{dim}_H(\Lambda_\varepsilon\cap W^s(x, f))\to -\frac{h_\mu(f)}{\lambda_2(\mu)}$ as $\varepsilon\to 0$.
Since $\text{dim}E^u = \text{dim}E^s =1$, the local product structure is a Lipschitz homeomorphism with Lipschitz inverse.
Therefore
\begin{equation}\label{dimension of hyperbolic set}
\text{dim}_H\Lambda_\varepsilon= t^u+t^s.
\end{equation}
Then it follows from (\ref{Young1}) and (\ref{usdimension}) that
\begin{equation}\label{approximation}
\text{dim}_H\Lambda_\varepsilon \to \text{dim}_H\mu\ (\text{as } \varepsilon\to0).
\end{equation}
The question is whether the Hausdorff dimension for hyperbolic sets $\Lambda_\varepsilon$ can approximate  the Hausdorff dimension for the SRB measure $\mu$ in higher dimension ? The following theorem extends (\ref{approximation}) to the case of higher dimension. Here we assume the hyperbolic measure $\mu$ is an SRB measure and has one dimensional stable manifold for every $x\in\Gamma$.

\begin{Main}\label{Main B}
Let $f: M\rightarrow M$ be a $C^{1+\alpha}$ diffeomorphism on a $d_0$-dimensional compact Riemannian manifold $M$ and $\mu$ be a hyperbolic ergodic SRB $f$-invariant
probability measure. Assume $\mu$ has one dimensional stable manifold (i.e. $\dim E_{l+1}(x)=1$ for every $x\in\Gamma$ and $k=l+1$ in (\ref{tangent}) ). Then there exists a sequence of hyperbolic sets $\Lambda_{n}$ such that
\begin{align*}
	\dim_{H}\Lambda_{n}\rightarrow \dim_{H}\mu~(\text{as } n\to\infty).
\end{align*}
\end{Main}




For the second result, we mention that the local product structure is lack of enough regularity, which may be not a Lipschitz homeomorphism with Lipschitz inverse. Therefore the formula (\ref{dimension of hyperbolic set}) couldn't hold.
So our second result couldn't be deduced directly as in the surface case.
For the hyperbolic ergodic SRB measure $\mu$ in Theorem \ref{Main B}, Katok's horseshoe construction tells us that, there exists a sequence of horseshoes $\{\Lambda_\varepsilon\}_{\varepsilon>0}$, and the topology entropy of $f$ restricted to horseshoes can be arbitrarily close to the measure theoretic entropy. Moreover each horseshoe $\Lambda_\varepsilon$ has a dominated splitting, with approximately the same Lyapunov exponents on $\Lambda_\varepsilon$.
We define super-additive potentials $\{-\psi^t(\cdot, f^{2^n})\}_{n\geq1}$ (see the definition in (\ref{star2})) in the unstable direction.
Let $t_n$ be the root of Bowen's equation $P(f^{2^n}|_{\Lambda_\varepsilon}, -\psi^t(\cdot, f^{2^n}))=0$, and $\mu_n^u$ be the unique equilibrium state for $P(f^{2^n}|_{\Lambda_\varepsilon}, -\psi^{t_n}(\cdot, f^{2^n}))$.
It follows from Climenhaga, Pesin and Zelerowicz's result \cite{CliPZ} that the family of conditional measures $\{\mu_{n,x}^u\}_{x\in\Lambda_\varepsilon}$ of $\mu_n^u$ on the local unstable leaves has the $u$-Gibbs property for $-\psi^{t_n}(\cdot, f^{2^n})$.
We use some techniques in Theorem \ref{hyperbolic1} to estimate the lower/upper pointwise dimension of $\mu_{n,x}^u$ for every $x$ in $\Lambda_\varepsilon$. That is
$$t_n\leq \underline{d}_{\mu_{n,x}^u}(x)\leq \overline{d}_{\mu_{n,x}^u}(x)\leq u, \text{ for every } x\in\Lambda_\varepsilon$$
and $\displaystyle\lim_{\varepsilon\to0}\lim_{n\to\infty}t_n=u$, where $u$ is the dimension of the unstable manifold.
Similarly we can also define $\{\phi^t(\cdot, f^{2^n})\}_{n\geq1}$ in the stable direction.
Let $t_n'$ be the root of Bowen's equation $P(f^{2^n}|_{\Lambda_\varepsilon}, \phi^t(\cdot, f^{2^n}))=0$, and $\mu_n^s$ be the unique equilibrium state for $P(f^{2^n}|_{\Lambda_\varepsilon}, \phi^{t_n'}(\cdot, f^{2^n}))$.
Since the stable direction is one dimension, then we get $d_{\mu_{n,x}^s}(x)=t_n'$ for every $x\in\Lambda_\varepsilon$, and $\displaystyle\lim_{\varepsilon\to0}\lim_{n\to\infty}t_n'=\frac{h_\mu(f)}{-\lambda_{\ell+1}(\mu)}$, where $\{\mu_{n,x}^s\}_{x\in\Lambda_\varepsilon}$ is the family of conditional measures of $\mu_n^s$ on the local stable leaves, and $\lambda_{\ell+1}(\mu)$ is the unique negative Lyapunov exponent of $\mu$.
Secondly we construct measures $\hat{\mu}_n=\mu_{n,x}^u\times \mu_{n,x}^s$ on hyperbolic set $\Lambda_\varepsilon$ and use this measure to measure a ball so as to get a suitable lower and upper bound of the Hausdorff dimension of the horseshoe $\Lambda_\varepsilon$. This is inspired by the work in~\cite{BV}. Therefore we have $\displaystyle\lim_{\varepsilon\to0}\dim_H\Lambda_\varepsilon=u+\frac{h_\mu(f)}{-\lambda_{\ell+1}(\mu)}$.
On the other hand, Barreira, Pesin and Schmeling \cite{BPS} proved $\dim_H\mu=\dim_H^u\mu+\dim_H^s\mu$, here $\dim_H^u\mu$ and $\dim_H^s\mu$ are the pointwise dimensions along stable
and unstable local manifolds respectively for $\mu$ almost everywhere $x$.
And Ledrappier and Young presented $\dim_H^s\mu=\frac{h_\mu(f)}{-\lambda_{\ell+1}(\mu)}$, which is a well known result for dimension theory. Finally we use the properties of SRB measures and Hausdorff dimension to prove $\dim_H^u\mu=u$. This shows that the Hausdorff dimension of $\mu$ can be approximated by the Hausdorff dimension of horseshoes.

We arrange the paper as follows. In the preliminaries, we give some basic notions and properties about Hausdorff dimension, hyperbolic set, Markov partition, Gibbs measure, pointwise dimension, topological pressure and so on. The proof of our main results will appear in Section $3$.

\section{Preliminaries}

In this section, we recall some definitions and some preliminary results.
\subsection{Hausdorff dimension}
First we recall the definition of the Hausdorff dimension. For more introduction, one can refer to \cite{Pesin}.

Let $X$ be a compact metric space. Given  a subset $Z$ of $X$, for $s\geq 0$ and $\delta>0$, define
\[
\mathcal{H}_{\delta}^{s}(Z)\triangleq\inf \left\{\sum_{i}|U_i|^s: \
Z\subset \bigcup_{i}U_i,~|U_i|\leq \delta\right\}
\]
where $|\cdot|$ denotes the diameter of a set. The quantity
$\mathcal{H}^{s}(Z)\triangleq\lim\limits_{\delta\rightarrow 0}\mathcal{H}_{\delta}^{s}(Z)$ is called the {\em $s$-dimensional Hausdorff measure} of $Z$. Define the {\em Hausdorff dimension} of $Z$, denoted by $\dim_H  Z$, as follows:
\[
\dim_H  Z =\inf \{s:\ \mathcal{H}^{s}(Z)=0\}=\sup \{s: \mathcal{H}^{s}(Z)=\infty\}.
\]
One can check that the Hausdorff dimension satisfies the monotonicity, that is, for $Y_{1}\subset Y_{2}\subset X$, we have ~$\dim_{H}Y_{1}\leq \dim_{H}Y_{2}$.

The following two propositions are basic tools in dimension theory. One can refer to ~\cite{Pesin}.
\begin{Prop}\label{dimension of measure_upper}
	Assume that there are numbers $C>0,~d>0$ and a Borel finite measure $\mu$ on a measurable set $Z$ such that for every $x\in Z$ and $r>0$, it holds that
	$$\mu(B(x,r))\geq Cr^{d},$$
	then $dim_{H}Z\leq d$.
\end{Prop}

\begin{Prop}\label{dimension of measure_lower}
	Assume that there are numbers $C>0,~d>0$ and a Borel finite measure $\mu$ on a measurable set $Z$ such that for $\mu-$almost every $x\in Z$ and $r>0$, it holds that
	$$\mu(B(x,r))\leq Cr^{d},$$
	then $dim_{H}Z\geq d$.
\end{Prop}
This proposition is known as the {\it mass distribution principle}.

Given a Borel probability measure $\mu$ on $X$, {\em the Hausdorff dimension of the measure $\mu$} is defined as
$$\dim_{H}\mu=\inf\{\dim_{H}Y:Y\subset X ,~\mu(Y)=1\}.$$
{\em The lower and upper pointwise dimension of $\mu$ at the point $x\in X$} are defined by
\begin{equation*}
\underline{d}_\mu(x)=\liminf_{r\to 0}\frac{\log\mu(B(x,r))}{\log r} \text{ and } \overline{d}_\mu(x)=\limsup_{r\to 0}\frac{\log\mu(B(x,r))}{\log r}
\end{equation*}
where $B(x,r)$ denotes the ball of radius $r$ centered at $x$. The quantity  $\dim_H\mu$ can indeed be defined in terms of a local quantity. Namely \cite{BW2006}
\begin{equation}\label{adofhdm}
\dim_H\mu=\text{ess sup}\{\underline{d}_\mu(x): x\in X\},
\end{equation}
here the essential supremum is taken with respect to $\mu$. In particular, if there exists a number $d$ such that
$$\lim_{r\to0}\frac{\log\mu(B(x,r))}{\log r}=d$$
for $\mu$ almost every $x\in X$, then $\dim_H\mu=d$. This criterion was established by Young in \cite{Y}. The limit when it exists is called {\em the pointwise dimension of $\mu$ at $x$}. We also recall the following statement, which relates the Hausdorff dimension with the lower pointwise dimension $\underline{d}_\mu$.
\begin{Prop}\label{ehdms}
The following properties hold:
\begin{itemize}
  \item [ (1) ] if $\underline{d}_\mu(x)\geq\alpha$ for $\mu$-almost every $x\in X$, then $\dim_H\mu\geq\alpha$;
  \item [ (2) ] if $\underline{d}_\mu(x)\leq\alpha$ for every $x\in Z\subseteq X$, then $\dim_HZ\leq\alpha$.
\end{itemize}
\end{Prop}

\subsection{Hyperbolic set, Markov partition and Gibbs measures}
Let $M$ be a compact Riemannian manifold. Suppose $f$ is a $C^{1+\alpha}$ diffeomorphism on $M$ and $\Lambda\subset M$ is an $f$-invariant compact set. We say $\Lambda$ is a {\it hyperbolic set } if for any $x\in \Lambda$, the tangent space admits a decomposition $T_{x}M=E^{s}(x)\oplus E^{u}(x)$ with the following properties:
\begin{enumerate}
	\item The splitting is $Df$-invariant: $D_{x}f E^{\sigma}(x)=E^{\sigma}(f(x))$ for $\sigma=s,u$.
	\item The stable subspace $E^{s}(x)$ is uniformly contracting and the unstable subspace $E^{u}(x)$ is uniformly expanding: there are constants $C\geq 1$ and $0<\chi<1$ such that for every $n\geq 0$ and $v^{s,u}\in E^{s,u}(x)$, we have $$\|D_{x}f^{n}v^{s}\|\leq C\chi^{n}\|v^{s}\|\quad \text{and} \quad \|D_{x}f^{-n}v^{u}\|\leq C\chi^{n}\|v^{u}\|.$$
\end{enumerate}

Recall that a hyperbolic set $\Lambda$ is {\it locally maximal} if there exists an open neighborhood $U$ of $\Lambda$ such that $\Lambda=\bigcap_{n\in\mathbb{Z}}f^{n}(U)$. $f$ is called {\it topologically transitive} on $\Lambda$, if for any two nonempty (relative) open subsets $U,V\subset\Lambda$ there exists $n>0$ such that $f^n(U)\cap V\neq\varnothing$.
Assume that $\Lambda$ is a locally maximal hyperbolic set and $f$ is topologically transitive on $\Lambda$. A finite cover $\mathcal{P}=\{P_{1},P_{2},...,P_{l}\}$ of $\Lambda$ is called a {\it Markov partition} of $\Lambda$ with respect to $f$ if
\begin{enumerate}
	\item $P_{i}\neq \varnothing$ and $P_{i}=\overline{\text{int } P_{i}}$ for each $i=1,2,...,l$,
	\item $\text{int } P_{i} \cap \text{int } P_{j} =\varnothing$ whenever $i\neq j$,
	\item for each $x\in \text{int } P_{i}\cap f^{-1} (\text{int } P_{j})$, we have
	$$f(W^{s}_{\text{loc}}(x, f)\cap P_{i})\subset W^{s}_{\text{loc}}(f(x), f)\cap P_{j},$$
	$$f(W^{u}_{\text{loc}}(x, f)\cap P_{i})\supset W^{u}_{\text{loc}}(f(x), f)\cap P_{j},$$
\end{enumerate}
here the topology used for the interior is the induced topology on $\Lambda$. The elements in the Markov partition are called {\it rectangles}.

Let $\mathcal{P}=\{P_{1},P_{2},...,P_{l}\}$ be a Markov partition of $\Lambda$ with diameter as small as desired (we refer to \cite{KH} for details and references). We equip the space of sequences $\Sigma_{l}=\{1,2,...,l\}^{\mathbb{Z}}$ with the distance
\begin{align*}
	d(w,w^{\prime})\triangleq \sum_{j\in\mathbb{Z}}e^{-|j|}|i_{j}-i^{\prime}_{j}|,
\end{align*}
where $w=(\cdots i_{-1}i_{0}i_{1}\cdots)$ and $w^{\prime}=(\cdots i^{\prime}_{-1}i^{\prime}_{0}i^{\prime}_{1}\cdots)$. With this distance, $\Sigma_l$ becomes a compact metric space. We also consider the shift map $\sigma: \Sigma_{l}\rightarrow \Sigma_{l}$ defined by $(\sigma(w))_{j}=i_{j+1}$ for each $w=(\cdots i_{-1}i_{0}i_{1}\cdots)\in \Sigma_{l}$ and $j\in\mathbb{Z}$. The restriction of $\sigma$ to the set
\begin{align*}
	\Sigma_{A}\triangleq\{(\cdots i_{-1}i_{0}i_{1}\cdots)\in \Sigma_{l}:a_{i_{j}i_{j+1}}=1 \text{~for all~} j\in\mathbb{Z}\}
\end{align*}
is the two-sided {\it topological Markov chain} with the transition matrix $A=(a_{ij})$ of the Markov partition (i.e. $a_{ij}=1$ if $\text{int } P_{i}\cap f^{-1}(\text{int } P_{j})\neq \varnothing$ and $a_{ij}=0$ otherwise). It is easy to show that one can define a coding map $\pi:\Sigma_{A}\rightarrow \Lambda$ of the hyperbolic set $\Lambda$ by
\begin{align*}
	\pi(\cdots i_{-n}\cdots i_{0}\cdots i_{n}\cdots)=\bigcap_{n\in \mathbb{Z}}f^{-n}(P_{i_{n}}).
\end{align*}
The map $\pi$ is surjective and satisfies $\pi\circ\sigma=f\circ\pi$.

A sequence $\mathbf{i}=(i_{-m}\cdots i_{0}\cdots i_{n})$, where $1\leq i_{j}\leq l$, is called {\it admissible} if for $j=-m,...,0,...,n-1$, it holds that $\text{int } P_{j}\cap f^{-1}(\text{int } P_{j+1})\neq \varnothing$. Given an admissible sequence $\mathbf{i}=(i_{-m}\cdots i_{0}\cdots i_{n})$, one can define the cylinder
$$P_{i_{-m}\cdots i_{0}\cdots i_{n}}=\bigcap_{j=-m}^{n}f^{-j}(P_{i_{j}}).$$

We say a Borel probability measure $\mu$ on $\Lambda$ is a {\it Gibbs measure} for a continuous function $\varphi$ on $\Lambda$ if there exists $C>0$ such that for every $n\in \mathbb{N}$, every admissible sequence $(i_{0}i_{1}\cdots i_{n-1})$ and $x\in P_{i_{0}i_{1}\cdots i_{n-1}}$, we have
\begin{align*}
	C^{-1}\leq \frac{\mu(P_{i_{0}i_{1}\cdots i_{n-1}})}{\exp(-nP+S_{n}\varphi(x))}\leq C,
\end{align*}
where $P$ is a constant and $S_{n}\varphi(x)=\sum_{j=0}^{n-1}\varphi(f^{j}(x))$.

The following result was established by Climenhaga, Pesin and Zelerowicz in \cite{CliPZ}, which enables us to get a lower bound of the Hausdorff dimension of the intersection of the local unstable leaf and a hyperbolic set with dominated splitting.
\begin{Prop}\label{u-Gibbs}
	Suppose $\Lambda$ is a locally maximal hyperbolic set for a $C^{1+\alpha}$ diffeomorphism $f$ and $f$ is topologically transitive on $\Lambda$. Assume $\varphi:\Lambda\rightarrow \mathbb{R}$ is a H$\ddot{o}$lder potential. Then there exists a unique equilibrium state $\mu$ for $\varphi$. And the family of the conditional measures $\{\mu_{x}^{u}\}_{x\in\Lambda}$ of $\mu$ on the local unstable leaf has the $u$-Gibbs property for the potential $\varphi$, that is, there is $Q>0$ such that
	for every  $x\in\Lambda$, if $\pi(\cdots i_{-m}\cdots i_{0}\cdots i_{n}\cdots)=x$, then for each $n\in \mathbb{N}$ and $y\in W^{u}_{\text{loc}}(x, f)\cap P_{i_{0}\cdots i_{n-1}}$, we have
	\begin{align*}
		Q^{-1}\leq \frac{\mu_{x}^{u}(W^{u}_{\text{loc}}(x, f)\cap P_{i_{0}\cdots i_{n-1}})}{\exp(-nP(\varphi)+S_{n}\varphi(y))}\leq Q,
	\end{align*}
	where $P(f, \varphi)$ is the topological pressure.
\end{Prop}

\begin{Rem}\label{remark}
	This Proposition is a direct result of Corollary $4.9$, Theorem $4.10$ and $(3)$ of Theorem $4.11$ in \cite{CliPZ}. For readers' convenience, we give a short explanation. In \cite{CliPZ}, this result was stated in terms of Bowen balls. But there is no essential difference between Bowen balls and cylinders. Besides, in \cite{CliPZ}, the u-Gibbs property was built for a sequence of reference measures $\{m_{x}^{C}\}_{x\in\Lambda}$ and they obtained the equivalence of the reference measures and the conditional measures along the local unstable leaves for almost every point. But we can still deduce the u-Gibbs property of the conditional measure for every point. In detail, for $x\in \Lambda$ such that $\mu_{x}^{u}$ is equivalent to $m_{x}^{C}$, we choose a suitable size of neighborhood $U(x)$ of $x$. Here the size only depends on the size of the Markov partition and the local product property of $\Lambda$. For $y\in U(x)\setminus W^{u}_{\text{loc}}(x, f)$, define
	$\mu_{y}^{u}=\mu_{x}^{u}\circ\pi_{xy}^{-1}$, where $\pi_{xy}$ is the holonomy map from $W_{\text{loc}}^{u}(x, f)$ to $W_{\text{loc}}^{u}(y, f)$. The u-Gibbs property of the measure $\mu_{y}^{u}$ follows from Corollary $4.9$, Theorem $4.10$ and $(3)$ of Theorem $4.11$ in \cite{CliPZ}. By the Gibbs property of the measure $\mu$, $\mu(B)>0$ for every open ball $B$. This gives the u-Gibbs property of the measure $\mu_{y}^{u}$ for every $y\in \Lambda$.
\end{Rem}

\subsection{Approximation of hyperbolic measures by hyperbolic sets with dominated splittings}
First we recall the definition of the dominated splitting.
Let $M$ be a $d_{0}$-dimensional compact smooth Riemannian manifold and $f:M\rightarrow M$ be a $C^{1+\alpha}$ diffeomorphism. Suppose $\Lambda\subset M$ is a compact $f$-invariant set. We say $\Lambda$ admits a {\it dominated splitting} if there is a continuous invariant splitting $T_{\Lambda}M=E\oplus F$ and constants $C>0,\lambda\in (0,1)$ such that for each $x\in \Lambda$, $n\in \mathbb{N}$, $0\neq u\in E(x)$ and $0\neq v\in F(x)$, it holds that
\begin{align*}
\frac{\|D_{x}f^{n}(u)\|}{\|u\|}\leq C\lambda^{n}\frac{\|D_{x}f^{n}(v)\|}{\|v\|}.
\end{align*}
We write $E\preceq F$ if $F$ dominates $E$. Furthermore, given $0<\ell\leq d_{0}$, we say a continuous invariant splitting $T_{\Lambda}M=E_{1}\oplus...\oplus E_{\ell}$ dominated if there are numbers $\lambda_{1}<\lambda_{2}<...<\lambda_{\ell}$, constants $C>0$ and $0<\varepsilon<\min_{1\leq i\leq\ell-1}\{\frac{\lambda_{i+1}-\lambda_{i}}{100}\}$ such that for every $x\in \Lambda$, $n\in\mathbb{N}$ and $1\leq j\leq \ell$ and each unit vector $u\in E_{j}(x)$, it holds that
\begin{align*}
C^{-1}\exp[n(\lambda_{j}-\varepsilon)]\leq \|D_{x}f^{n}(u)\|\leq C \exp[n(\lambda_{j}+\varepsilon)].
\end{align*}
In particular, $E_{1}\preceq...\preceq E_{\ell}$. We shall use the notion $\{\lambda_{j}\}$-dominated when we want to stress the dependence on the numbers $\{\lambda_{j}\}$.

The authors in \cite{katok1980} and \cite{ACW} gave the proof of the following result, which enables us to approximate a hyperbolic measure by hyperbolic sets with dominated splittings.
\begin{Thm}\cite{katok1980,ACW}\label{Katok}
Let $f$ be a $C^{1+\alpha}$ diffeomorphism on a compact Riemannian manifold $M$ and $\mu$ be a hyperbolic ergodic invariant measure for $f$ with $h_{\mu}(f)>0$. Then for any $\varepsilon>0$ and a weak$^{\ast}$ neighborhood $\mathcal{V}$ of $\mu$ in the space of $f$-invariant probability measures on $M$, there exists an $f$-invariant compact subset $\Lambda_{\varepsilon}\subset M$ such that
\begin{enumerate}[\quad \quad$(1)$]
\item $\Lambda_{\varepsilon}$ is $\varepsilon$-close to the support set of $\mu$ in the Hausdorff distance,
\item $h_{\text{top}}(f|_{\Lambda_{\varepsilon}})\geq h_{\mu}(f)-\varepsilon$,
\item all the invariant probability measures supported on $\Lambda_{\varepsilon}$ lie in $\mathcal{V}$,
\item there is  a $\{\lambda_{j}(\mu)\}$-dominated splitting $TM=E_{1}\oplus E_{2}\oplus...\oplus E_{\ell}$ over $\Lambda_{\varepsilon}$, where $\lambda_{1}(\mu)<...<\lambda_{\ell}(\mu)$ are distinct Lyapunov exponents of $f$ with respect to the measure $\mu$.
\end{enumerate}
\end{Thm}

\subsection{Classical topological pressure and super-additive topological pressure}
First we recall the definition of the topological pressure. For more information, one can refer to \cite{Pesin, WALTERS}.
Let $X$ be a compact metric space equipped with the metric $d$ and $f:X\rightarrow X$ be a continuous map. Given $n\in\mathbb{N}$, define the metric $d_{n}$ on $X$ as
\begin{align*}
	d_{n}(x,y)=\max\{d(f^{k}(x),f^{k}(y)):0\leq k<n\}.
\end{align*}
Given $r>0$, denote by $B_{n}(x,r)=\{y:d_{n}(x,y)<r\}$ the {\it Bowen balls}. We say a set $E\subset X$ is $(n,r)${\it -separated} for $X$ if $d_n(x,y)>\varepsilon$ for any two distinct points $x,y\in E$.
A sequence of continuous potential functions $\Phi=\{\varphi_n\}_{n\geq1}$ is called {\it sub-additive}, if
\begin{equation*}
\varphi_{m+n}\leq\varphi_n+\varphi_m\circ f^n, \text{ for all } m,n\geq1.
\end{equation*}
Similarly, we call a sequence of continuous potential functions $\Psi=\{\psi_n\}_{n\geq1}$ {\it super-additive} if $-\Psi=\{-\psi_n\}_{n\geq1}$ is sub-additive.

Let $\Phi=\{\varphi_n\}_{n\geq1}$ be a sub-additive sequence of continuous potentials on $X$, set
\begin{equation*}
P_n(\Phi, \varepsilon)=\sup\{\sum_{x\in E}e^{\varphi_n(x)}: E \text{ is an } (n,\varepsilon) \text{-separated subset of } X\}.
\end{equation*}
The quantity
\begin{equation*}
P(f,\Phi)=\lim_{\varepsilon\to0}\limsup_{n\to\infty}\frac1n\log P_n(\Phi,\varepsilon)
\end{equation*}
is called the {\it sub-additive topological pressure} of $\Phi$.
The sub-additive topological pressure satisfies the following  variational principle, see \cite{cfh}.
\begin{Thm}\label{variational principle}
Let $\Phi=\{\varphi_n\}_{n\geq1}$ be a sub-additive sequence of continuous potentials on a compact dynamical system $(X,f)$. Then
\begin{equation*}
	P(f,\Phi)=\sup\{h_{\mu}(f)+\mathcal{F}_*(\Phi,\mu)| ~\mu\in\mathcal{M}_{inv}(f), \mathcal{F}_*(\Phi,\mu)\neq-\infty\},
\end{equation*}
where $\mathcal{F}_*(\Phi,\mu)=\lim_{n\to\infty}\frac1n\int\varphi_n d\mu$ (the equality is due to the standard sub-additive argument), and $\mathcal{M}_{inv}(f)$ is the set of all $f$-invariant Borel probability measures on X.
\end{Thm}

\begin{Rem}
If $\varPhi=\{\varphi_n\}_{n\geq1}$ is \emph{additive} in the sense that $\varphi_n(x)=\varphi(x)+\varphi(fx)+\cdots+\varphi(f^{n-1}x)\triangleq S_n\varphi(x)$ for some continuous function $\varphi: X\to \mathbb{R}$, we simply denote the topological pressure $P(f, \varPhi)$ as $P(f, \varphi)$.
\end{Rem}

Though it is unknown whether the variational principle holds for super-additive topological pressure, Cao, Pesin and Zhao gave an alternative definition via variational principle in \cite{caopesinzhao2019}.
Given a sequence of super-additive continuous potentials $\Psi=\{\psi_{n}\}_{n\geq 1}$ on a compact dynamical system $(X, f)$,
the {\it super-additive topological pressure} of $\Psi$ is defined as
\begin{align*}
	P_{\text{var}}(f,\Psi)\triangleq \sup \{h_{\mu}(f)+\mathcal{F}_{\ast}(\Psi,\mu):\mu\in \mathcal{M}(f), \mathcal{F}_*(\Psi,\mu)\neq-\infty\}.
\end{align*}
where 
$$\mathcal{F}_{\ast}(\Psi,\mu)=\lim_{n\rightarrow \infty}\frac{1}{n}\int \psi_{n}~d\mu=\sup_{n\in \mathbb{N}} \frac{1}{n}\int \psi_{n}~d\mu.$$
The second equality is due to the standard sub-additive argument.
It was proved in \cite{caopesinzhao2019} that the following relation between the super-additive topological pressure and the topological pressure for additive potentials holds.
\begin{Prop}\label{relation}
	Let $\Psi=\{\psi_{n}\}_{n\geq 1}$ be a  sequence of super-additive continuous potentials on $X$. Then it holds that
	\begin{align*}
		P_{\text{var}}(f,\Psi)=\lim_{n\rightarrow\infty} P(f,\frac{\psi_{n}}{n})=\lim_{n\rightarrow \infty}\frac{1}{n}P(f^{n},\psi_{n}).
	\end{align*}
\end{Prop}

\section{Proof of the main results}

\subsection{Proof of Theorem~\ref{hyperbolic1}}
We split the proof of the theorem into two steps.

 Let $M$ be a $d_{0}$-dimensional compact Riemannian manifold and $f:M\rightarrow M$ be a $C^{1+\alpha}$ diffeomorphism. For each $x\in M$, the following quantities
$$ \|D_xf\|=\sup_{0\neq u\in T_xM}\frac{\|D_xf(u)\|}{\|u\|}
\text{ and } m(D_xf)=\inf_{0\neq u\in T_xM}\frac{\|D_xf(u)\|}{\|u\|}
$$
are respectively called \emph{the maximal norm and minimum norm of the differentiable operator $D_xf: T_xM \to T_{fx}M$}, where $\|\cdot\|$ is the norm induced by the Riemannian metric on $M$.
Suppose $\Lambda\subset M$ is a hyperbolic set, and the diffeomorphism $f|_{\Lambda}$ possesses a $\{\lambda_{j}\}$-dominated splitting $T_{\Lambda}M=E_{1}\oplus E_{2}\oplus...\oplus E_{k}$ with $E_{1}\succeq E_{2}\succeq ...\succeq E_{k}$ and $\lambda_{1}>\lambda_{2}>...>\lambda_{\ell}>0>\lambda_{\ell+1}>...>\lambda_{k}$. Let $m_{j}=\text{dim} E_{j}$, denote $r_0=0$ and
$r_{j}=m_{1}+...+m_{j} \text{ for } j\in\{1,2,...,\ell\}$.
Set $u=m_{1}+...+m_{\ell}$.
For $s\in [0,u]$, $n\geq 1$ and $x\in \Lambda$, define
\begin{align}\label{star2}
\psi^{s}(x,f^{n})=\sum_{j=1}^{d}m_{j}\log\|D_{x}f^{n}|_{E_{j}}\|+(s-r_{d})\log\|D_{x}f^{n}|_{E_{d+1}}\|
\end{align}
if $r_{d}\leq s<r_{d+1}$ for some $d\in \{0,1,...,\ell-1\}$. Here we set $\sum_{j=1}^{0}m_{j}\log\|D_{x}f^{n}|_{E_{j}}\|=0$. It is clear that
\begin{equation}\label{star3}
\Psi_{f}(s)\triangleq \{-\psi^{s}(x,f^{n})\}_{n\geq 1}
\end{equation}
is super-additive.


\subsubsection{Lower bound of the Hausdorff dimension for the hyperbolic set $\Lambda$ with dominated splitting on the local unstable leaf}

Using the additive topological pressure of $-\psi^{s}(\cdot,f)$, we can give lower bound of the Hausdorff dimension for the hyperbolic set $\Lambda$ on the local unstable leaf as follows:

\begin{Lem}\label{lower bound 1}
Let $f:M\rightarrow M$ be a $C^{1+\alpha}$ diffeomorphism on a $d_{0}$-dimensional compact Riemannian manifold $M$, and $\Lambda\subset M$ be a hyperbolic set.
Assume that the diffeomorphism $f|_{\Lambda}$ possesses a $\{\lambda_{j}\}$-dominated splitting $T_{\Lambda}M=E_{1}\oplus E_{2}\oplus...\oplus E_{k}$ with $E_{1}\succeq E_{2}\succeq ...\succeq E_{k}$ and $\lambda_{1}>\lambda_{2}>...>\lambda_{\ell}>0>\lambda_{\ell+1}>...>\lambda_{k}$. Then there exists a constant $\bar{a}>0$ such that for any small $r>0$ and any $x\in \Lambda$, one has
$$\mu_x^u(B^u(x,r))\leq \bar{a}r^{s_1}$$
and
$$\text{dim}_{H}(\Lambda\cap W^{u}_{\text{loc}}(x, f))\geq s_{1},$$
where $s_{1}$ is the unique root of Bowen's equation $P(f|_{\Lambda},-\psi^{s}(\cdot,f))=0$, $\mu$ is the unique equilibrium state for $P(f|_{\Lambda},-\psi^{s_{1}}(\cdot,f))$, and $\{\mu_{x}^{u}\}_{x\in\Lambda}$ is the family of the conditional measures of $\mu$ on the local unstable leaves.
\end{Lem}

\begin{proof}
Let $\{P_{1},P_{2},...,P_{l}\}$ be a Markov partition of $\Lambda$. Since $\Lambda$ is locally maximal, there exists $\delta>0$ such that for each $i=1,2,...,l$, the closed $\delta$-neighborhood $\tilde{P}_{i}$ of $P_{i}$ satisfies $\tilde{P}_{i}\subset U$, where $U$ is an open neighborhood of $\Lambda$  in the definition of the locally maximal hyperbolic set, and $\tilde{P}_{i}\cap \tilde{P}_{j}=\varnothing$ whenever $P_{i}\cap P_{j}=\varnothing$. Given an admissible sequence $\mathbf{i}=\cdots i_{-m}\cdots i_{0}i_{1}\cdots i_{n-1}$, a point $x_{\mathbf{i}}\in\bigcap_{j=-\infty}^{n-1}f^{-j}(P_{i_{j}})$ and the cylinder $W^{u}_{\text{loc}}(x_{\mathbf{i}}, f) \cap P_{i_{0}i_{1}\cdots i_{n-1}}$, we denote by $W^{u}_{\text{loc}}(x_{\mathbf{i}}, f) \cap \tilde{P}_{i_{0}i_{1}\cdots i_{n-1}}$ the corresponding closed $\delta$-neighborhood of $W^{u}_{\text{loc}}(x_{\mathbf{i}}, f) \cap P_{i_{0}i_{1}\cdots i_{n-1}}$. Note that the $\{\lambda_{j}\}$-dominated splitting can be extended to $U$ and the map $x\mapsto E_{i}(x)$ is H$\ddot{\text{o}}$lder continuous on $\Lambda$ and also can be extended to $U$. So is the map $x\mapsto \|D_{x}f|_{E_{i}}\|$ for each $i=1,2,...,k.$

Since $\psi^{s_{1}}(\cdot,f)$ is H$\ddot{\text{o}}$lder continuous on $\Lambda$, there is a unique equilibrium state $\mu$ for $P(f|_{\Lambda},-\psi^{s_{1}}(\cdot,f))$, which is a Gibbs measure. Proposition ~\ref{u-Gibbs} tells us that the family of the conditional measures $\{\mu_{x}^{u}\}_{x\in\Lambda}$ of $\mu$ on the local unstable leaves have the $u$-Gibbs property for $-\psi^{s_{1}}(\cdot,f)$. For any $x\in\Lambda$, there is some $(\cdots\bar{i}_{-m}\cdots\bar{i}_{0}\cdots\bar{i}_{n}\cdots)\in\Sigma_{A}$ with  $\pi(\cdots\bar{i}_{-m}\cdots\bar{i}_{0}\cdots\bar{i}_{n}\cdots)=x$. Then there exists $K>0$ such that for each $n\in \mathbb{N}$ and each $y\in W^{u}_{\text{loc}}(x, f)\cap P_{\bar{i}_{0}\bar{i}_{1}\cdots\bar{i}_{n-1}}$, we have
\begin{eqnarray*}
 \begin{aligned}
K^{-1}\exp(-\sum_{i=0}^{n-1}\psi^{s_{1}}(f^{i}(y),f))&\leq \mu_{x}^{u}(W^{u}_{\text{loc}}(x, f)\cap P_{\bar{i}_{0}\bar{i}_{1}\cdots\bar{i}_{n-1}})\\
                                                      &\leq K \exp(-\sum_{i=0}^{n-1}\psi^{s_{1}}(f^{i}(y),f)).
 \end{aligned}
\end{eqnarray*}
Since $T_{\Lambda}M=E_{1}\oplus E_{2}\oplus...\oplus E_{k}$ is a dominated splitting, the angles between different subspaces $E_{i}$ are uniformly bounded away from zero. Therefore there exists $a>0$ such that for each admissible sequence $\mathbf{i}=\cdots i_{-m}\cdots i_{0}i_{1}\cdots i_{n-1}$, there exists $\xi_{i}\in W^{u}_{\text{loc}}(x_{\mathbf{i}}, f)\cap \tilde{P}_{i_{0}i_{1}\cdots i_{n-1}}$ for each $i=1,2,...,\ell$ with a rectangle of sides
\begin{align*}
\overbrace{a\|D_{\xi_{1}}f^{n}|_{E_{1}}\|^{-1},...,a\|D_{\xi_{1}}f^{n}|_{E_{1}}\|^{-1},}^{m_{1}}...,\overbrace{a\|D_{\xi_{\ell}}f^{n}|_{E_{\ell}}\|^{-1},...,a\|D_{\xi_{\ell}}f^{n}|_{E_{\ell}}\|^{-1}}^{m_{\ell}}
\end{align*}
contained in $W^{u}_{\text{loc}}(x_{\mathbf{i}}, f)\cap \tilde{P}_{i_{0}i_{1}\cdots i_{n-1}}$. Since $f$ is expanding along the unstable direction and the map $x\mapsto\|D_{x}f|_{E_{i}}\|^{-1}$  is H$\ddot{\text{o}}$lder continuous for $i=1,2,...,\ell$, there exists $C_{0}>0$ such that for any $W^{u}_{\text{loc}}(x_{\mathbf{i}}, f)\cap \tilde{P}_{i_{0}i_{1}\cdots i_{n-1}}$ and any $y, z \in W^{u}_{\text{loc}}(x_{\mathbf{i}}, f)\cap \tilde{P}_{i_{0}i_{1}\cdots i_{n-1}}$, any $i\in \{1,2,...,\ell\}$, we have
\begin{align*}
\frac{1}{C_{0}}\leq \frac{\prod_{j=0}^{n-1}\|D_{f^{j}(y)}f|_{E_{i}}\|^{-1}}{\prod_{j=0}^{n-1}\|D_{f^{j}(z)}f|_{E_{i}}\|^{-1}}\leq C_{0}.
\end{align*}
Notice that $\|D_{\xi_{i}}f^{n}|_{E_{i}}\|^{-1}\geq \prod_{j=0}^{n-1}\|D_{f^{j}(\xi_{i})}f|_{E_{i}}\|^{-1}$. Thus there exists $a_{1}>0$ and $y\in W^{u}_{\text{loc}}(x_{\mathbf{i}}, f)\cap P_{i_{0}i_{1}\cdots i_{n-1}}$ such that $W^{u}_{\text{loc}}(x_{\mathbf{i}}, f)\cap \tilde{P}_{i_{0}i_{1}\cdots i_{n-1}}$ contains a rectangle of sides
\begin{align*}
\overbrace{a_{1}A_{1}(y,n),...,a_{1}A_{1}(y,n),}^{m_{1}}\overbrace{a_{1}A_{2}(y,n),...,a_{1}A_{2}(y,n),}^{m_{2}}...,\\
\overbrace{a_{1}A_{\ell}(y,n),...,a_{1}A_{\ell}(y,n)}^{m_{\ell}},
\end{align*}
where $A_{i}(y,n)=\prod_{j=0}^{n-1}\|D_{f^{j}(y)}f|_{E_{i}}\|^{-1}$ for each $i\in \{1,2,...,\ell\}$.
Notice that there is $i\in \{1,2,...,\ell\}$ such that $r_{i}\leq s_{1}<r_{i+1}$. Fix $r>0$ small enough and set
\begin{align*}
\mathcal{Q}=\{\mathbf{i}=(\cdots\bar{i}_{-m}\cdots\bar{i}_{0}i_{1}\cdots i_{n-1}):a_{1}A_{i+1}(y,n)\leq r\ \text{for}\ \text{all}\ y \in W^{u}_{\text{loc}}(x, f)\cap P_{\bar{i}_{0}i_{1}\cdots i_{n-1}};\\
\text{but}~a_{1}A_{i+1}(z,n-1)>r~\text{for}~\text{some}~z \in W^{u}_{\text{loc}}(x, f)\cap P_{\bar{i}_{0}i_{1}\cdots i_{n-1}}\}.
\end{align*}
Therefore, for each $\mathbf{i}=(\cdots\bar{i}_{-m}\cdots\bar{i}_{0}i_{1}\cdots i_{n-1})\in \mathcal{Q}$, we have
\begin{align*}
br<a_{1}A_{i+1}(y,n)\leq r ~\text{for} ~\text{all}~y\in W^{u}_{\text{loc}}(x, f)\cap P_{\bar{i}_{0}i_{1}\cdots i_{n-1}},
\end{align*}
where $b=C_{0}^{-1}\min_{x\in \Lambda}\|D_{x}f\|^{-1}$.

Let $B\subset W^{u}_{\text{loc}}(x, f)$ be a ball of radius $r$ and $\tilde{B}\subset W^{u}_{\text{loc}}(x, f)$ a ball of radius $Cr$ with the same center as that of $B$, where $C$ is a constant only depended on the dimension of the unstable manifold. Put
\begin{align*}
\mathcal{Q}_{1}=\{\mathbf{i}\in \mathcal{Q}|P_{\mathbf{i}}\cap B\neq \varnothing\}.
\end{align*}
Hence, for each $\mathbf{i}\in\mathcal{Q}_{1}$, we have $\tilde{P}_{\mathbf{i}}\cap \tilde{B}$ contains a rectangle of sides
\begin{align*}
\overbrace{a_{1}A_{1}(y,n),...,a_{1}A_{1}(y,n)}^{m_{1}},...,\overbrace{a_{1}A_{i}(y,n),...,a_{1}A_{i}(y,n)}^{m_{i}},\\
\overbrace{a_{1}A_{i+1}(y,n),...,a_{1}A_{i+1}(y,n)}^{u-r_{i}}.
\end{align*}
It follows that
\begin{align*}
a_{1}^{u}A_{i+1}(y,n)^{u-r_{i}}A_{i}(y,n)^{m_{i}}...A_{1}(y,n)^{m_{1}}\leq vol (\tilde{P}_{\mathbf{i}}\cap \tilde{B}),
\end{align*}
where $vol$ denotes the volume induced on the local unstable leaf.

Since
\begin{align*}
A_{i+1}(y,n)^{u-r_{i}}&=A_{i+1}(y,n)^{s_{1}-r_{i}}A_{i+1}(y,n)^{u-s_{1}}\\
&\geq A_{i+1}(y,n)^{s_{1}-r_{i}}(\frac{b}{a_{1}})^{u-s_{1}}r^{u-s_{1}},
\end{align*}
we have
\begin{align*}
a_{2}r^{u-s_{1}}A_{i+1}(y,n)^{s_{1}-r_{i}}A_{i}(y,n)^{m_{i}}...A_{1}(y,n)^{m_{1}}\leq vol (\tilde{P}_{\mathbf{i}}\cap \tilde{B})
\end{align*}
for some constant $a_{2}>0$. Therefore,
\begin{align*}
\sum_{\mathbf{i}\in\mathcal{Q}_{1}}a_{2}r^{u-s_{1}}A_{i+1}(y,n)^{s_{1}-r_{i}}A_{i}(y,n)^{m_{i}}...A_{1}(y,n)^{m_{1}}\leq vol (\tilde{B})\leq C^{u}a_{3}r^{u}
\end{align*}
for some constant $a_{3}>0$. Hence,
\begin{align*}
\sum_{\mathbf{i}\in\mathcal{Q}_{1}}A_{i+1}(y,n)^{s_{1}-r_{i}}A_{i}(y,n)^{m_{i}}...A_{1}(y,n)^{m_{1}}\leq a_{4}r^{s_{1}}
\end{align*}
for some constant $a_{4}>0$.
On the other hand, we have
\begin{align*}
\mu_{x}^{u}(B)&\leq \sum_{(\cdots\bar{i}_{-m}\cdots\bar{i}_{0}i_{1}\cdots i_{n-1})\in\mathcal{Q}_{1}}\mu_{x}^{u}(W^{u}_{\text{loc}}(x, f)\cap P_{\bar{i}_{0}i_{1}\cdots i_{n-1}})\\
&\leq K \sum_{(\cdots\bar{i}_{-m}\cdots\bar{i}_{0}i_{1}\cdots i_{n-1})\in\mathcal{Q}_{1}} \exp(-\sum_{i=0}^{n-1}\psi^{s_{1}}(f^{i}(y),f))\\
&=K \sum_{(\cdots\bar{i}_{-m}\cdots\bar{i}_{0}i_{1}\cdots i_{n-1})\in\mathcal{Q}_{1}} A_{i+1}(y,n)^{s_{1}-r_{i}}A_{i}(y,n)^{m_{i}}...A_{1}(y,n)^{m_{1}}\\
&\leq a_{5}r^{s_{1}}
\end{align*}
for some constant $a_{5}>0$. The second inequality is due to the $u$-Gibbs property for $-\psi^{s_{1}}(\cdot,f)$ of $\mu_{x}^{u}$. This implies that $\text{dim}_{H}\mu_{x}^{u}\geq s_{1}$ and hence $\text{dim}_{H}(\Lambda\cap W^{u}_{\text{loc}}(x, f))\geq \dim_{H}\mu_{x}^{u}\geq s_{1}$.
\end{proof}

The following Lemma shows that the zero of the super-additive topological pressure of $\Psi_f(s)$ provides a sharper lower bound of the Hausdorff dimension for $\Lambda$ on the local unstable leaf.

\begin{Lem}\label{lower bound 2}
Let $f:M\rightarrow M$ be a $C^{1+\alpha}$ diffeomorphism on a $d_{0}$-dimensional compact Riemannian manifold $M$, and $\Lambda\subset M$ be a hyperbolic set.
Assume that the diffeomorphism $f|_{\Lambda}$ possesses a $\{\lambda_{j}\}$-dominated splitting $T_{\Lambda}M=E_{1}\oplus E_{2}\oplus...\oplus E_{k}$ with $E_{1}\succeq E_{2}\succeq ...\succeq E_{k}$ and $\lambda_{1}>\lambda_{2}>...>\lambda_{\ell}>0>\lambda_{\ell+1}>...>\lambda_{k}$. Then for any $x\in \Lambda$, we have
$$\text{dim}_{H}(\Lambda\cap W^{u}_{\text{loc}}(x, f))\geq s^{\ast},$$
where $s^{\ast}$ is the unique root of Bowen's equation $P_{\text{var}}(f|_{\Lambda},\Psi_{f}(s))=0$.
\end{Lem}

\begin{proof}
For each $n\in \mathbb{N}$, $\Lambda$ is also a hyperbolic set for $f^{2^{n}}$. Therefore $W^u_{\text{loc}}(x, f)=W^u_{\text{loc}}(x, f^{2^{n}})$ for $x\in\Lambda$. By Lemma \ref{lower bound 1}, for $x\in \Lambda$, we have $\text{dim}_{H}(\Lambda\cap W^{u}_{\text{loc}}(x, f))\geq s_{n}$, where $s_{n}$ is the unique root of the equation
\begin{align*}
P(f^{2^{n}}|_{\Lambda},-\psi^{s}(\cdot,f^{2^{n}}))=0.
\end{align*}

One can show that $s_{n}\leq s_{n+1}$. In detail, each $f^{2^{n}}$-invariant measure is also $f^{2^{n+1}}$-invariant. This together with the super-additivity of $\{-\psi^{s}(\cdot,f^{n})\}_{n\geq 1}$ yields that for any $f^{2^{n}}$-invariant measure $\mu$ supported on $\Lambda$, it holds that
\begin{align*}
\frac{1}{2^{n+1}}P(f^{2^{n+1}}|_{\Lambda},-\psi^{s}(\cdot,f^{2^{n+1}}))&\geq \frac{1}{2^{n+1}}\left(h_{\mu}(f^{2^{n+1}})-2\int \psi^{s}(x,f^{2^{n}})~d\mu\right)\\
&=\frac{1}{2^{n}}\left(h_{\mu}(f^{2^{n}})-\int \psi^{s}(x,f^{2^{n}})~d\mu\right).
\end{align*}
By the variational principle of the topological pressure of the additive potential, it holds that
\begin{align*}
\frac{1}{2^{n+1}}P(f^{2^{n+1}}|_{\Lambda},-\psi^{s}(\cdot,f^{2^{n+1}}))\geq \frac{1}{2^{n}}P(f^{2^{n}}|_{\Lambda},-\psi^{s}(\cdot,f^{2^{n}})).
\end{align*}
It follows that $s_{n}\leq s_{n+1}$. And hence the limit $s^{\ast}\triangleq \lim_{n\rightarrow \infty}s_{n}$ exists. Therefore $\text{dim}_{H}(\Lambda\cap W^{u}_{\text{loc}}(x, f))\geq s^{\ast}$ for $x\in \Lambda$.

Next we show that $s^{\ast}$ is the root of $P_{\text{var}}(f|_{\Lambda},\Psi_{f}(s))=0$.
By the continuity of $P_{\text{var}}(f|_{\Lambda},\Psi_{f}(s))$ with respect to $s$, it holds that $$P_{\text{var}}(f|_{\Lambda},\Psi_{f}(s^{\ast}))=\lim_{n\rightarrow \infty}P_{\text{var}}(f|_{\Lambda},\Psi_{f}(s_{n})).$$ By Proposition \ref{relation}, for each $n\in \mathbb{N}$, it holds that $$P_{\text{var}}(f|_{\Lambda},\Psi_{f}(s_{n}))=\lim_{k\rightarrow \infty}\frac{1}{2^{k}}P(f^{2^{k}}|_{\Lambda},-\psi^{s_{n}}(\cdot,f^{2^{k}})).$$ By the monotonicity of $\{s_{n}\}$,  for each  $n\in \mathbb{N}$, it holds that
$$\lim_{k\rightarrow \infty}\frac{1}{2^{k}}P(f^{2^{k}}|_{\Lambda},-\psi^{s_{n}}(\cdot,f^{2^{k}}))\geq 0.$$ It follows that $P_{\text{var}}(f|_{\Lambda},\Psi_{f}(s_{n}))\geq 0$ and hence $P_{\text{var}}(f|_{\Lambda},\Psi_{f}(s^{\ast}))\geq 0$.
On the other hand,
$$P_{\text{var}}(f|_{\Lambda},\Psi_{f}(s^{\ast}))=\lim_{k\rightarrow \infty}\frac{1}{2^{k}}P(f^{2^{k}}|_{\Lambda},-\psi^{s^{\ast}}(\cdot,f^{2^{k}}))\leq 0.$$
Thus $s^{\ast}$ is the root of the equation $P_{\text{var}}(f|_{\Lambda},\Psi_{f}(s))=0$.
This completes the proof of the lemma.
\end{proof}

\subsubsection{Proof of the theorem}


Let $\Gamma$ be the set of points which are regular in the sense of Oseledec \cite{Oseledets} with respect to the measure $\mu$. For  every $x\in\Gamma$, denote its distinct Lyapunov exponents by
\begin{align*}
	\lambda_{1}(\mu)>...>\lambda_{\ell}(\mu)>0>\lambda_{\ell+1}(\mu)>...>\lambda_{k}(\mu)
\end{align*}
with multiplicities $m_{1}, m_{2},...,m_{k}\geq 1$	and let
\begin{align*}
	T_{x}M=E_{1}(x)\oplus E_{2}(x)\oplus...\oplus E_{k}(x)
\end{align*}
be the corresponding decomposition of its tangent space, where $0<k\leq d_{0}$. Denote $r_0=0$ and $r_j=m_1+m_2+...+m_j$ for $j=1,2,...,k$. Recall $\psi^s(\cdot, f^n)$ and $\Psi_f(s)$ as (\ref{star2}) and (\ref{star3}) respectively. Let $\lambda_j'(\mu)=\lambda_{d+1}(\mu)$ if $r_d\leq j<r_{d+1}$ for $d=\{0,1,...,k-1\}$, and $\lambda'_{d_0}(\mu)=\lambda_{k}(\mu)$.

Since $\mu$ is an ergodic SRB measure for $f$, by Pesin's entropy formula ~\cite{LS}, it holds that
\begin{equation}\label{positive entropy}
h_{\mu}(f)-\sum^{\ell}_{i=1} m_i \lambda_{i}(\mu)=0.
\end{equation}
Therefore, we obtain $h_{\mu}(f)>0$. Set $u=m_1+m_2+...+m_\ell$. Fixed $0<\varepsilon<\frac{\lambda_{\ell}(\mu)}{2(3u+1)}$ and a neighborhood $\mathcal{V}_{\varepsilon}$ of $\mu$, by Theorem~\ref{Katok}, there exists a hyperbolic set $\Lambda_{\varepsilon}$ such that
\begin{enumerate}[(i)]
\item $\Lambda_{\varepsilon}$ admits a $\{\lambda_{j}(\mu)\}$-dominated splitting $T_{\Lambda_{\varepsilon}}M=E_{1}\oplus E_{2}\oplus...\oplus E_{k}$;
\item $h_{\text{\text{top}}}(f|_{\Lambda_{\varepsilon}})\geq h_{\mu}(f)-\varepsilon$;
\item $|\lambda_{j}'(\mu)-\lambda_{j}(\nu)|\leq \varepsilon$ for each ergodic Borel probability measure $\nu$ supported on $\Lambda_{\varepsilon}$ and $j=1,2,...,d_0$;
\item any $f$-invariant Borel probability measure $\nu$ supported on $\Lambda_{\varepsilon}$ lies in $\mathcal{V}_{\varepsilon}$.
\end{enumerate}
Since $f$ is hyperbolic on $\Lambda_{\varepsilon}$, there exists an ergodic $f$-invariant probability measure $\nu$ supported on $\Lambda_{\varepsilon}$ such that
\begin{equation}\label{star4}
h_{\nu}(f|_{\Lambda_{\varepsilon}})=h_{\text{top}}(f|_{\Lambda_{\varepsilon}}).
\end{equation}
Denote $\nu$ by $\nu_{\varepsilon}$ if we need to emphasize the dependence on $\varepsilon$.

It follows from Lemma \ref{lower bound 2} and the monotonicity of the Hausdorff dimension that
$$s_{\varepsilon}\leq \text{dim}_{H}(\Lambda_{\varepsilon}\cap W^{u}_{\text{loc}}(x, f))\leq u$$
for $x\in \Lambda_{\varepsilon}$, where $s_{\varepsilon}$ is the unique root of the equation $P_{\text{var}}(f|_{\Lambda_{\varepsilon}},\Psi_{f}(s))=0.$ 
If $s_\varepsilon=u$, then the proof of the theorem is complete.
If $s_\varepsilon<u$, then there is some $d\in \{0,1,...,\ell-1\}$ such that  $r_{d}\leq s_{\varepsilon}<r_{d+1}$. By (i) and (iii), one has
\begin{align*}
	&\lim_{n\rightarrow \infty}\frac{1}{n}\int\psi^{s_{\varepsilon}}(x,f^{n})~d\nu\\
	=&\sum_{j=1}^{d}m_{j}\lim_{n\rightarrow \infty}\frac{1}{n}\int\log\|D_{x}f^{n}|_{E_{j}}\|~d\nu+(s_{\varepsilon}-r_{d})\lim_{n\rightarrow \infty}\frac{1}{n}\int\log\|D_{x}f^{n}|_{E_{d+1}}\|~d\nu\\
	\leq &\sum_{j=1}^{d}m_{j}\lambda_{j}(\mu)+(s_{\varepsilon}-r_{d})\lambda_{d+1}(\mu)+s_{\varepsilon}\varepsilon\\
	\leq &\sum_{i=1}^{[s_{\varepsilon}]}\lambda_{i}(\nu)+(s_{\varepsilon}-[s_{\varepsilon}])\lambda_{[s_{\varepsilon}]+1}(\nu)+2s_{\varepsilon}\varepsilon.
\end{align*}
Therefore
\begin{align}
0 =& P_{var}(f|_{\Lambda_\varepsilon}, \Psi_f(s_\varepsilon)) \nonumber \\
 \geq &  h_{\nu}(f|_{\Lambda_{\varepsilon}})-\lim_{n\rightarrow \infty}\frac{1}{n}\int \psi^{s_{\varepsilon}}(x,f^{n})~d\nu \nonumber   \\
 \geq & h_{\nu}(f|_{\Lambda_{\varepsilon}})-\sum_{i=1}^{[s_{\varepsilon}]}\lambda_{i}(\nu)-(s_{\varepsilon}-[s_{\varepsilon}])\lambda_{[s_{\varepsilon}]+1}(\nu)-2s_{\varepsilon}\varepsilon. \label{entropy2}
\end{align}
Then it follows from (ii), (iii), (\ref{star4}) and (\ref{entropy2}) that
\begin{align*}
\sum^{\ell}_{i=1} m_i \lambda_{i}(\mu)&=h_{\mu}(f)\\
&\leq h_{\nu}(f|_{\Lambda_{\varepsilon}})+\varepsilon\\
&\leq \sum_{i=1}^{[s_\varepsilon]}\lambda_{i}(\nu)+(s_{\varepsilon}-[s_\varepsilon])\lambda_{[s_{\varepsilon}]+1}(\nu)+2s_{\varepsilon}\varepsilon+\varepsilon\\
&\leq \sum_{i=1}^{d+1}m_i\lambda_{i}(\mu)-(r_{d+1}-s_{\varepsilon})\lambda_{d+1}(\mu)+3s_{\varepsilon}\varepsilon+\varepsilon\\
&< \sum_{i=1}^{d+1}m_i\lambda_{i}(\mu)-(r_{d+1}-s_{\varepsilon})\lambda_{d+1}(\mu)+\frac12\lambda_{\ell}(\mu).
\end{align*}
Therefore $r_{\ell-1}\leq s_{\varepsilon}<r_{\ell}$.
By Lemma~\ref{lower bound 2}, (\ref{entropy2}), (ii), (iii), (\ref{star4}) and (\ref{positive entropy}), one has
\begin{align}
\text{dim}_{H} (\Lambda_{\varepsilon}\cap W^{u}_{\text{loc}}(x, f))&\geq s_{\varepsilon} \nonumber \\
&\geq \frac{h_{\nu}(f)-\sum_{i=1}^{[s_{\varepsilon}]}\lambda_{i}(\nu)+[s_{\varepsilon}]\lambda_{[s_{\varepsilon}]+1}(\nu)-2u\varepsilon}{\lambda_{[s_{\varepsilon}]+1}(\nu)} \nonumber \\
&\geq \frac{h_{\mu}(f)-\varepsilon-\sum_{i=1}^{\ell-1}m_i\lambda_{i}(\mu)+r_{\ell-1}\lambda_{\ell}(\mu)-4u\varepsilon}{\lambda_{\ell}(\mu)+\varepsilon} \nonumber \\
&=\frac{u\lambda_{\ell}(\mu)-(4u+1)\varepsilon}{\lambda_{\ell}(\mu)+\varepsilon} \nonumber \\
&=u-\frac{(5u+1)\varepsilon}{\lambda_{\ell}(\mu)+\varepsilon}, \label{lowerbofroot}
\end{align}
here the second inequality is by (\ref{entropy2}), and the third inequality is by (ii), (iii) and (\ref{star4}).
Hence
$$u\geq \text{dim}_{H}(\Lambda_{\varepsilon}\cap W^{u}_{\text{loc}}(x, f)) \geq u-\frac{(5u+1)\varepsilon}{\lambda_{\ell}(\mu)+\varepsilon}$$
for $x\in \Lambda_{\varepsilon}$. Letting $\varepsilon$ tend to zero and shrinking $\mathcal{V}_{\varepsilon}$, we have that
$$\lim_{\varepsilon\rightarrow 0}\text{dim}_{H} (\Lambda_{\varepsilon}\cap W^{u}_{\text{loc}}(x, f))=u$$
uniformly for $x\in \Lambda_{\varepsilon}$ and $\nu_{\varepsilon}\rightarrow \mu$ in the weak$^{\ast}$ topology.
This completes the proof of Theorem \ref{hyperbolic1}.





\subsection{Upper bound of the Hausdorff dimension for the hyperbolic set $\Lambda$ with dominated splitting on the local unstable leaf.}\label{upper}
 In this subsection, we still assume that the diffeomorphism $f|_{\Lambda}$ possesses a $\{\lambda_{j}\}$-dominated splitting as that in Lemma \ref{lower bound 1}. The following lemma is prepared for Section $3.3$.

Let $m_{j}=\text{dim} E_{j}$, $r^{\prime}_{j}=m_{\ell}+...+m_{\ell-j+1}$ for $j\in\{1,2,...,\ell\}$ and $r^{\prime}_{0}=0$. Set $u=m_{1}+...+m_{\ell}$. For $t\in [0,u]$, $n\geq 1$ and $x\in \Lambda$, define
\begin{align*}
	\hat{\psi}^{t}(x,f^{n})=\sum_{j=\ell-d+1}^{\ell}m_{j}\log m(D_{x}f^{n}|_{E_{j}})+(t-r^{\prime}_{d})\log m(D_{x}f^{n}|_{E_{\ell-d}})
\end{align*}
if $r^{\prime}_{d}\leq t< r^{\prime}_{d+1}$ for some $d\in \{0,1,...,\ell-1\}$. Here we set $\sum_{j=\ell+1}^{\ell}m_{j}\log m(D_{x}f^{n}|_{E_{j}})=0$ and $\hat{\psi}^u(x, f^n)=\sum_{j=1}^\ell m_j \log m(D_xf^n|_{E_j})$. 

\begin{Lem}\label{upper bound 1}
Let $f:M\rightarrow M$ be a $C^{1+\alpha}$ diffeomorphism on a $d_{0}$-dimensional compact Riemannian manifold $M$, and $\Lambda\subset M$ be a hyperbolic set.
Assume that the diffeomorphism $f|_{\Lambda}$ possesses a $\{\lambda_{j}\}$-dominated splitting $T_{\Lambda}M=E_{1}\oplus E_{2}\oplus...\oplus E_{k}$ with $E_{1}\succeq E_{2}\succeq ...\succeq E_{k}$ and $\lambda_{1}>\lambda_{2}>...>\lambda_{\ell}>0>\lambda_{\ell+1}>...>\lambda_{k}$. Then there exists a constant $\tilde{a}>0$ such that for any $r>0$ small and any $x\in \Lambda$, we have
$$\mu_{x}^{u}(B^{u}(x,r))\geq \tilde{a}r^{t_{1}}$$
	and $$\text{dim}_{H}(\Lambda\cap W^{u}_{\text{loc}}(x,f))\leq t_{1},$$
	where $t_{1}$ is the unique root of Bowen's equation $P(f|_{\Lambda},-\hat{\psi}^{t}(\cdot,f))=0$, $\mu$ is the unique equilibrium state for $P(f|_{\Lambda},-\hat{\psi}^{t_{1}}(\cdot,f))$, and $\{\mu_{x}^{u}\}_{x\in\Lambda}$ is the family of the conditional measures of $\mu$ on the local unstable leaves.
\end{Lem}

\begin{proof}
	Given a Markov partition $\{P_{1},P_{2},...,P_{l}\}$ of $\Lambda$ and $\delta>0$ as that in the proof of Lemma~\ref{lower bound 1}. We still denote by $\tilde{P}_{i}$ the $\delta$-neighborhood of $P_{i}$ as that in the proof of Lemma~\ref{lower bound 1}. Note that the map $x\mapsto m(D_{x}f|_{E_{i}})$ can be extended to $U$ as a H$\ddot{\text{o}}$lder continuous map for each $i=1,2,...,k,$ where $U$ is a neighborhood of $\Lambda$ from the definition of the locally maximal hyperbolic set.

	Since $\hat{\psi}^{t_{1}}(\cdot,f)$ is H$\ddot{\text{o}}$lder continuous on $\Lambda$, there is a unique equilibrium state $\mu$ for $P(f|_{\Lambda},-\hat{\psi}^{t_{1}}(\cdot,f))$, which is a Gibbs measure. Proposition ~\ref{u-Gibbs} tells us that the family of the conditional measures $\{\mu_{x}^{u}\}_{x\in\Lambda}$ of $\mu$ on the local unstable leaves have the $u$-Gibbs property for $-\hat{\psi}^{t_{1}}(\cdot,f)$. For any $x\in\Lambda$, there is some $(\cdots\bar{i}_{-m}\cdots\bar{i}_{0}\cdots\bar{i}_{n}\cdots)\in\Sigma_{A}$ with  $\pi(\cdots\bar{i}_{-m}\cdots\bar{i}_{0}\cdots\bar{i}_{n}\cdots)=x$. Then there exists $\hat{K}>0$ such that for each $n\in \mathbb{N}$ and each $y\in W^{u}_{\text{loc}}(x,f)\cap P_{\bar{i}_{0}\bar{i}_{1}\cdots\bar{i}_{n-1}}$, we have
\begin{eqnarray*}
 \begin{aligned}
	\hat{K}^{-1}\exp(-\sum_{i=0}^{n-1}\hat{\psi}^{t_{1}}(f^{i}(y),f))&\leq \mu_{x}^{u}(W^{u}_{\text{loc}}(x,f)\cap P_{\bar{i}_{0}\bar{i}_{1}\cdots\bar{i}_{n-1}})\\
                                                                     &\leq \hat{K} \exp(-\sum_{i=0}^{n-1}\hat{\psi}^{t_{1}}(f^{i}(y),f)).
 \end{aligned}
	\end{eqnarray*}
Given an admissible sequence $\mathbf{i}=\cdots i_{-m}\cdots i_{0}i_{1}\cdots i_{n-1}$ and any $x_{\mathbf{i}}\in\bigcap_{j=-\infty}^{n-1}f^{-j}(P_{i_{j}})$, we denote by $W^{u}_{\text{loc}}(x_{\mathbf{i}}, f) \cap \tilde{P}_{i_{0}i_{1}\cdots i_{n-1}}$ the corresponding closed $\delta$-neighborhood of $W^{u}_{\text{loc}}(x_{\mathbf{i}}, f) \cap P_{i_{0}i_{1}\cdots i_{n-1}}$.	
Since $T_{\Lambda}M=E_{1}\oplus E_{2}\oplus...\oplus E_{k}$ is a dominated splitting, the angles between different subspaces $E_{i}$ are uniformly bounded away from zero. Therefore there exists $\hat{a}>0$ such that for each admissible sequence $\mathbf{i}=\cdots i_{-m}\cdots i_{0}i_{1}\cdots i_{n-1}$, there exists $\xi_{i}\in W^{u}_{\text{loc}}(x_{\mathbf{i}},f)\cap \tilde{P}_{i_{0}i_{1}\cdots i_{n-1}}$ for each $i=1,2,...,\ell$,  $W^{u}_{\text{loc}}(x_{\mathbf{i}},f)\cap \tilde{P}_{i_{0}i_{1}\cdots i_{n-1}}$ is contained in a rectangle of sides
	\begin{align*}
		\overbrace{\hat{a}\times m(D_{\xi_{1}}f^{n}|_{E_{1}})^{-1},...,\hat{a}\times  m(D_{\xi_{1}}f^{n}|_{E_{1}})^{-1},}^{m_{1}}...,\\
		\overbrace{\hat{a}\times  m(D_{\xi_{\ell}}f^{n}|_{E_{\ell}})^{-1},...,\hat{a}\times  m(D_{\xi_{\ell}}f^{n}|_{E_{\ell}})^{-1}}^{m_{\ell}}.
	\end{align*}
	Since $f$ is expanding along the unstable direction and the map $x\mapsto m(D_{x}f|_{E_{i}})^{-1}$  is H$\ddot{\text{o}}$lder continuous for $i=1,2,...,\ell$, then there exists $\hat{C}_{0}>0$ such that for any $W^{u}_{\text{loc}}(x_{\mathbf{i}},f)\cap \tilde{P}_{i_{0}i_{1}\cdots i_{n-1}}$ and any $y, z \in W^{u}_{\text{loc}}(x_{\mathbf{i}},f)\cap \tilde{P}_{i_{0}i_{1}\cdots i_{n-1}}$, any $i\in \{1,2,...,\ell\}$, we have
	\begin{align*}
		\frac{1}{\hat{C}_{0}}\leq \frac{\prod_{j=0}^{n-1}m(D_{f^{j}(y)}f|_{E_{i}})^{-1}}{\prod_{j=0}^{n-1}m(D_{f^{j}(z)}f|_{E_{i}})^{-1}}\leq \hat{C}_{0}.
	\end{align*}
	Notice that $m(D_{\xi_{i}}f^{n}|_{E_{i}})^{-1}\leq \prod_{j=0}^{n-1}m(D_{f^{j}(\xi_{i})}f|_{E_{i}})^{-1}$. Thus there exists $\hat{a}_{1}>0$ and $y\in W^{u}_{\text{loc}}(x_{\mathbf{i}},f)\cap P_{i_{0}i_{1}\cdots i_{n-1}}$ such that $W^{u}_{\text{loc}}(x_{\mathbf{i}},f)\cap \tilde{P}_{i_{0}i_{1}\cdots i_{n-1}}$ is contained in a rectangle of sides
	\begin{align*}
		\overbrace{\hat{a}_{1}A_{1}(y,n),...,\hat{a}_{1}A_{1}(y,n),}^{m_{1}}\overbrace{\hat{a}_{1}A_{2}(y,n),...,\hat{a}_{1}A_{2}(y,n),}^{m_{2}}...,\\
		\overbrace{\hat{a}_{1}A_{\ell}(y,n),...,\hat{a}_{1}A_{\ell}(y,n)}^{m_{\ell}},
	\end{align*}
	where $A_{i}(y,n)=\prod_{j=0}^{n-1}m(D_{f^{j}(y)}f|_{E_{i}})^{-1}$ for each $i\in \{1,2,...,\ell\}$.
	Notice that there is $i\in \{0,1,...,\ell-1\}$ such that $r^{\prime}_{i}\leq t_{1}<r^{\prime}_{i+1}$. Fix $r>0$ small enough and set
	\begin{align*}
		\hat{\mathcal{Q}}=\{\mathbf{i}=(\cdots\bar{i}_{-m}\cdots\bar{i}_{0}i_{1}\cdots i_{n-1}):\hat{a}_{1}A_{\ell-i}(y,n)\leq r\ \text{for}\ \text{all}\ y \in W^{u}_{\text{loc}}(x,f)\cap P_{\bar{i}_{0}i_{1}\cdots i_{n-1}};\\
		\text{but}~\hat{a}_{1}A_{\ell-i}(z,n-1)>r~\text{for}~\text{some}~z \in W^{u}_{\text{loc}}(x,f)\cap P_{\bar{i}_{0}i_{1}\cdots i_{n-1}}\}.
	\end{align*}
	Therefore, for each $\mathbf{i}=(\cdots\bar{i}_{-m}\cdots\bar{i}_{0}i_{1}\cdots i_{n-1})\in \hat{\mathcal{Q}}$, we have
	\begin{align*}
		\hat{b}r<\hat{a}_{1}A_{\ell-i}(y,n)\leq r ~\text{for} ~\text{all}~y\in W^{u}_{\text{loc}}(x,f)\cap P_{\bar{i}_{0}i_{1}\cdots i_{n-1}},
	\end{align*}
	where $\hat{b}=\hat{C}_{0}^{-1}\min_{x\in \Lambda}m(D_{x}f)^{-1}$.
	
	Let $B\subset W^{u}_{\text{loc}}(x,f)$ be a ball of radius $r$ and $\tilde{B}\subset W^{u}_{\text{loc}}(x,f)$ a ball of radius $\hat{C}r$ with the same center as that of $B$, where $\hat{C}$ is a constant only depended on the dimension of the unstable manifold. Put
	\begin{align*}
		\hat{\mathcal{Q}}_{1}=\{\mathbf{i}\in \hat{\mathcal{Q}}|P_{\mathbf{i}}\cap B\neq \varnothing\}.
	\end{align*}
	Hence, for each $\mathbf{i}\in\hat{\mathcal{Q}}_{1}$, we have $\tilde{P}_{\mathbf{i}}\cap \tilde{B}$ is contained in a rectangle of sides
	\begin{align*}
		\overbrace{\hat{a}_{1}A_{\ell-i}(y,n),...,\hat{a}_{1}A_{\ell-i}(y,n)}^{u-r^{\prime}_{i}}, \overbrace{\hat{a}_{1}A_{\ell-i+1}(y,n),...,\hat{a}_{1}A_{\ell-i+1}(y,n)}^{m_{\ell-i+1}},..., \\\overbrace{\hat{a}_{1}A_{\ell}(y,n),...,\hat{a}_{1}A_{\ell}(y,n)}^{m_{\ell}}.\\
	\end{align*}
	It follows that
	\begin{align*}
		\hat{a}_{1}^{u}A_{\ell-i}(y,n)^{u-r^{\prime}_{i}}A_{\ell-i+1}(y,n)^{m_{\ell-i+1}}...A_{\ell}(y,n)^{m_{\ell}}\geq vol (\tilde{P}_{\mathbf{i}}\cap \tilde{B}).
	\end{align*}
	Since
	\begin{align*}
		A_{\ell-i}(y,n)^{u-r^{\prime}_{i}}&=A_{\ell-i}(y,n)^{t_{1}-r^{\prime}_{i}}A_{\ell-i}(y,n)^{u-t_{1}}\\
		&\leq A_{\ell-i}(y,n)^{t_{1}-r^{\prime}_{i}}(\frac{r}{\hat{a}_{1}})^{u-t_{1}},
	\end{align*}
	then
	\begin{align*}
		\hat{a}_{2}r^{u-t_{1}}A_{\ell-i}(y,n)^{t_{1}-r^{\prime}_{i}}A_{\ell-i+1}(y,n)^{m_{\ell-i+1}}...A_{\ell}(y,n)^{m_{\ell}}\geq vol (\tilde{P}_{\mathbf{i}}\cap \tilde{B})
	\end{align*}
	for some constant $\hat{a}_{2}>0$, here $vol$ denotes the volume induced on the local unstable leaf. Therefore,
	\begin{align*}
		\sum_{\mathbf{i}\in\hat{\mathcal{Q}}_{1}}	\hat{a}_{2}r^{u-t_{1}}A_{\ell-i}(y,n)^{t_{1}-r^{\prime}_{i}}A_{\ell-i+1}(y,n)^{m_{\ell-i+1}}...A_{\ell}(y,n)^{m_{\ell}}\geq vol (\tilde{B})\geq \hat{C}^{u}\hat{a}_{3}r^{u}
	\end{align*}
	for some constant $\hat{a}_{3}>0$. Hence,
	\begin{align*}
		\sum_{\mathbf{i}\in\hat{\mathcal{Q}}_{1}}	A_{\ell-i}(y,n)^{t_{1}-r^{\prime}_{i}}A_{\ell-i+1}(y,n)^{m_{\ell-i+1}}...A_{\ell}(y,n)^{m_{\ell}}\geq \hat{a}_{4}r^{t_{1}}
	\end{align*}
	for some constant $\hat{a}_{4}>0$.
	On the other hand, we have
	\begin{align*}
		\mu_{x}^{u}(\tilde{B})&\geq \sum_{(\cdots\bar{i}_{-m}\cdots\bar{i}_{0}i_{1}\cdots i_{n-1})\in\hat{\mathcal{Q}}_{1}}\mu_{x}^{u}(W^{u}_{\text{loc}}(x,f)\cap P_{\bar{i}_{0}i_{1}\cdots i_{n-1}})\\
		&\geq \hat{K} \sum_{(\cdots\bar{i}_{-m}\cdots\bar{i}_{0}i_{1}\cdots i_{n-1})\in\hat{\mathcal{Q}}_{1}} \exp(-\sum_{i=0}^{n-1}\hat{\psi}^{t_{1}}(f^{i}(y),f))\\
		&=\hat{K} \sum_{(\cdots\bar{i}_{-m}\cdots\bar{i}_{0}i_{1}\cdots i_{n-1})\in\hat{\mathcal{Q}}_{1}} A_{\ell-i}(y,n)^{t_{1}-r^{\prime}_{i}}A_{\ell-i+1}(y,n)^{m_{\ell-i+1}}...A_{\ell}(y,n)^{m_{\ell}}\\
		&\geq \hat{a}_{5}r^{t_{1}}\\
        &= \hat{a}_{5}'(\hat{C} r)^{t_{1}}
	\end{align*}
	for some constant $\hat{a}_{5}>0$ and $\hat{a}'_5=\frac{\hat{a}_5}{(\hat{C})^{t_1}}$. For the first inequality, notice that it doesn't matter the Bowen balls in the expression of the Gibbs property are closed or open. Besides, the interior of the elements in the Markov partition are pairwise disjoint. Thus we can deal with the sets in the sum carefully such that the first inequality holds and the $u$-Gibbs property holds. The second inequality is due to the $u$-Gibbs property for $-\hat{\psi}^{t_{1}}(\cdot,f)$ of $\mu_{x}^{u}$. By Proposition~\ref{dimension of measure_upper}, we have $\text{dim}_{H}(\Lambda\cap W^{u}_{\text{loc}}(x,f))\leq t_{1}$.
\end{proof}

\subsection{Proof of Theorem~\ref{Main B}}

Let $\Gamma$ be the set of points which are regular in the sense of Oseledec \cite{Oseledets} with respect to the measure $\mu$. For  every $x\in\Gamma$, denote its distinct Lyapunov exponents by
\begin{align*}
	\lambda_{1}(\mu)>...>\lambda_{\ell}(\mu)>0>\lambda_{\ell+1}(\mu)
\end{align*}
with multiplicities $m_{1}, m_{2},...,m_{\ell+1}\geq 1$	and let
\begin{align*}
	T_{x}M=E_{1}(x)\oplus E_{2}(x)\oplus...\oplus E_{\ell+1}(x)
\end{align*}
be the corresponding decomposition of its tangent space, where $0<\ell+1\leq d_{0}$. Note that $\dim E_{\ell+1}=1$. Denote $r_0=0$ and $r_j=m_1+m_2+...+m_j$ for $j=1,2,...,\ell+1$.
Set $u=m_1+m_2+...+m_\ell$. For $t\in[0,u]$, $n\geq 1$ and $x\in\Gamma$, define
\begin{equation*}
\psi^t(x,f^n)=\sum_{j=1}^dm_j\log\|D_xf^n|_{E_j}\|+(t-r_d)\log\|D_xf^n|_{E_{d+1}}\|
\end{equation*}
if $r_d\leq t<r_{d+1}$ for some $d\in\{0,1,...,\ell-1\}$. Here we set $\psi^0(x,f^n)=0$ and $\psi^u(x,f^n)=\sum_{j=1}^\ell m_j\log\|D_xf^n|_{E_j}\|$. It is clear that
\begin{equation*}
\Psi_f(t)\triangleq\{-\psi^t(\cdot, f^n)\}_{n\geq1}
\end{equation*}
is super-additive. For $t'\in[0,1]$, $n\geq1$ and $x\in\Gamma$, define
\begin{equation*}
\phi^{t'}(x,f^n)=t'\log\|D_xf^n|_{E_{\ell+1}}\|.
\end{equation*}
Since $\dim E_{\ell+1}=1$, then
\begin{equation*}
\Phi_f(t')\triangleq\{\phi^{t'}(\cdot, f^n)\}_{n\geq1}
\end{equation*}
is additive. Let $\lambda_j'(\mu)=\lambda_{d+1}(\mu)$ if $r_d\leq j<r_{d+1}$ for $d=\{0,1,...,\ell\}$, and $\lambda'_{d_0}(\mu)=\lambda_{\ell+1}(\mu)$.

Fix any sufficiently small $\varepsilon\in(0,\frac{\lambda_\ell(\mu)}{2(3u+1)})$.
Since $\mu$ is a hyperbolic SRB measure, by Theorem \ref{Katok}, then there exists a hyperbolic set $\Lambda_\varepsilon$ satisfying
\begin{itemize}
  \item [ (a) ] $\Lambda_\varepsilon$ is contained in the $\varepsilon$-neighborhood of the support of $\mu$;
  \item [ (b) ] $|h_{\text{top}}(f|_{\Lambda_\varepsilon}) - h_\mu(f)|\leq\varepsilon$;
  \item [ (c) ] $d(\nu,\mu)<\varepsilon$ for every $f$-invariant Borel probability measure $\nu$ on $\Lambda_\varepsilon$;
  \item [ (d) ] $|\lambda_j'(\mu)-\lambda_j(\nu)|\leq\varepsilon$ for every ergodic $f$-invariant Borel probability measure $\nu$ supported on $\Lambda_\varepsilon$ and $j=1,2,...,d_0$;
  \item [ (e) ] $\Lambda_\varepsilon$ admits a $\{\lambda_j(\mu)\}$-dominated splitting $T_{\Lambda_\varepsilon}M=E_1\oplus E_2\oplus\cdots\oplus E_{\ell+1}$. And denote $E^u=E_1\oplus E_2\oplus\cdots\oplus E_\ell$ and $E^s=E_{\ell+1}$.
\end{itemize}
Fixed any integer $n\geq1$. Let $t_n$ be the root of Bowen's equation $P(f^{2^{n}}|_{\Lambda_\varepsilon},-\psi^t(\cdot,f^{2^n}))=0$, and $\mu_n^u$ be the unique equilibrium state for $P(f^{2^{n}}|_{\Lambda_\varepsilon},-\psi^{t_n}(\cdot,f^{2^n}))$. Proposition \ref{u-Gibbs} tells us that the family of conditional measures $\{\mu_{n,x}^u\}_{x\in\Lambda_\varepsilon}$ of $\mu^u_n$ on the local unstable leaves have the $u$-Gibbs property for $-\psi^{t_n}(\cdot,f^{2^n})$. As in the proof of Lemma \ref{lower bound 1} we obtain
\begin{equation}\label{pdofu}
t_n\leq\liminf_{r\to0}\frac{\log \mu^u_{n,x}(B^u(x,r))}{\log r}\leq\limsup_{r\to0}\frac{\log \mu^u_{n,x}(B^u(x,r))}{\log r}\leq u
\end{equation}
for every $x\in\Lambda_\varepsilon$, where $B^u(x,r)=\{y\in W^u_{\text{loc}}(x,f): d_u(x, y)<r\}$.
(Here $d_u$ is the metric induced by the Riemannian structure on the unstable manifold $W^u$ and $d_s$ is the metric induced by the Riemannian structure on the stable manifold $W^s$.)
Let $t_n'$ be the root of Bowen's equation
$$P(f^{2^{n}}|_{\Lambda_\varepsilon},\phi^t(\cdot,f^{2^n}))=0,$$
and $\mu_n^s$ be the unique equilibrium state for $P(f^{2^{n}}|_{\Lambda_\varepsilon},\phi^{t_n'}(\cdot,f^{2^n}))$.
Then Lemma \ref{lower bound 1} and Lemma \ref{upper bound 1} tell us that for every $x\in\Lambda_\varepsilon$,
\begin{equation}\label{pdofs}
\lim_{r\to0}\frac{\log\mu^s_{n,x}(B^s(x,r))}{\log r}=t_n',
\end{equation}
where $\{\mu_{n,x}^s\}_{x\in\Lambda_\varepsilon}$ is the family of conditional measures of of $\mu^s_n$ on the local stable leaves and $B^s(x,r)=\{y\in W^s_{\text{loc}}(x,f): d_s(x, y)<r\}$.

Define a measure $\hat{\mu}_n$ on $\Lambda_\varepsilon$ satisfying
\begin{equation*}
\hat{\mu}_n(B(x,r))=\mu^u_{n,x}(B^u(x,r))\times \mu^s_{n,x}(B^s(x,r))
\end{equation*}
for every $x\in\Lambda_\varepsilon$ and each $r$  much smaller than the size of the local unstable and stable leaf.
Fixed any $x\in\Lambda_\varepsilon$, for the above $\varepsilon>0$, (\ref{pdofu}) and (\ref{pdofs}) tell us that there is small $r_0>0$ such that for every $r\in(0, r_0)$,
\begin{equation}\label{star5}
r^{u+\varepsilon}\leq \mu_{n,x}^u(B^u(x,r))\leq r^{t_n-\varepsilon} \text{ and } r^{t_n'+\varepsilon}\leq \mu_{n,x}^s(B^s(x,r))\leq r^{t_n'-\varepsilon}.
\end{equation}
Note that there exist numbers $\gamma_1>1$ and $0<\gamma_2<1$ such that for any $x\in\Lambda_\varepsilon$, and sufficiently small $r>0$
\begin{equation}\label{star6}
B^u(x, \gamma_2r)\times B^s(x,\gamma_2r)\subseteq B(x,r) \subseteq B^u(x, \gamma_1r)\times B^s(x,\gamma_1r).
\end{equation}
Therefore, by (\ref{star5}) and (\ref{star6}) one has
\begin{eqnarray*}
\begin{aligned}
\hat{\mu}_n(B(x,\frac1\gamma_1r))\geq& \mu^u_{n,x}(B^u(x,\frac{\gamma_2}{\gamma_1}r))\cdot\mu^s_{n,x}(B^s(x,\frac{\gamma_2}{\gamma_1}r))\\
\geq& \left(\frac{\gamma_2}{\gamma_1}r\right)^{u+\varepsilon} \cdot \left(\frac{\gamma_2}{\gamma_1}r\right)^{t_n'+\varepsilon}
\end{aligned}
\end{eqnarray*}
and
\begin{eqnarray*}
\begin{aligned}
\hat{\mu}_n(B(x,\frac1\gamma_1r))\leq& \mu^u_{n,x}(B^u(x,r))\cdot\mu^s_{n,x}(B^s(x,r))\\
\leq& r^{t_n+t_n'-2\varepsilon}.
\end{aligned}
\end{eqnarray*}
This yields that
\begin{equation*}
t_n+t_n'-2\varepsilon\leq\underline{d}_{\hat{\mu}_n}(x)\leq\overline{d}_{\hat{\mu}_n}(x)\leq u+t_n'+2\varepsilon
\end{equation*}
for every $x\in\Lambda_\varepsilon$.
Combining with Proposition \ref{ehdms}, we obtain
\begin{equation}\label{star7}
t_n+t_n'-2\varepsilon\leq\dim_H\Lambda_\varepsilon\leq u+t_n'+2\varepsilon
\end{equation}
for every integer $n\geq 1$.

As in the proof of Lemma \ref{lower bound 2}, $t_\varepsilon:=\lim_{n\to\infty}t_n$ is the root of $P_{var}(f|_{\Lambda_\varepsilon}, \Psi_f(t))=0$. Combining with (\ref{lowerbofroot}) we have $\lim_{\varepsilon\to0}t_\varepsilon=u$. Therefore
\begin{equation}\label{star8}
\lim_{\varepsilon\to0}\lim_{n\to\infty}t_n=u.
\end{equation}
Since $\Phi_f(t')$ is additive, which is also super-additive, then we also obtain $t_\varepsilon'=\lim_{n\to\infty}t_n'$ is the root of $P_{var}(f|_{\Lambda_\varepsilon}, \Phi_f(t'))=0$.
It follows from Theorem \ref{variational principle}, (b) and (d) that
\begin{eqnarray*}
\begin{aligned}\frac{h_\mu(f)+\varepsilon}{-\lambda_{\ell+1}(\mu)-\varepsilon}\geq &     \,\,\,t_\varepsilon'\\
=&\sup\{\frac{h_\nu(f)}{-\lambda_{\ell+1}(\nu)}|\ \nu \text{ is ergodic } f\text{-invariant Borel probability measure}\\
&\ \ \ \ \ \ \ \ \ \ \ \ \ \ \ \ \ \ \ \ \ \text{ supported on } \Lambda_\varepsilon\}\\
\geq& \frac{h_\mu(f)-\varepsilon}{-\lambda_{\ell+1}(\mu)+\varepsilon}.
\end{aligned}
\end{eqnarray*}
Hence
\begin{equation}\label{star9}
\lim_{\varepsilon\to0}\lim_{n\to\infty}t_n'=\frac{h_\mu(f)}{-\lambda_{\ell+1}(\mu)}.
\end{equation}
Then (\ref{star7}), (\ref{star8}) and (\ref{star9}) tell us that
\begin{equation}\label{limofHD}
\lim_{\varepsilon\to0}\dim_H\Lambda_\varepsilon=u+\frac{h_\mu(f)}{-\lambda_{\ell+1}(\mu)}.
\end{equation}

Since $\dim E_{\ell+1}=1$ and $\mu$ is ergodic, Ledrappier and Young \cite{LY} proved that there is a constant $\dim_H^s\mu$ such that the pointwise dimension of $\mu$ on $W^s$-manifolds
\begin{equation*}
\lim_{r\to0}\frac{\log\mu_x^s(B^s(x,r))}{\log r}=\dim_H^s\mu
\end{equation*}
for $\mu$ almost every $x$, and here $\{\mu_x^s\}$ is a system of the conditional measures associated with $\xi^s$, which is a measurable partition of $M$ subordinate to the stable manifolds $W^s$. In fact they also proved that
\begin{equation}\label{star10}
\dim_H^s\mu=\frac{h_\mu(f)}{-\lambda_{\ell+1}(\mu)}.
\end{equation}

\begin{Lem}\label{u}
	Suppose $\mu$ is a hyperbolic ergodic SRB measure for a $C^{1+\alpha}$ diffeomorphism $f$, then for $\mu$ almost every point $x$,
\begin{equation*}
\lim_{r\to0}\frac{\log\mu_x^u(B^u(x,r))}{\log r}=u,
\end{equation*}
where $u$ is the dimension of the unstable subspace, $\{\mu_x^u\}$ is a system of the conditional measures associated with $\xi^u$, which is a measurable partition of $M$ subordinate to the unstable manifolds $W^u$.
\end{Lem}
\begin{proof}
  Fix any point $x$ such that $\mu_x^u$ is well defined. For any measurable set $B\subset W^{u}_{\text{loc}}(x,f)$ with $\mu^{u}_{x}(B)=1$, since $\mu$ is an SRB measure, then $Leb^{u}(B)>0$, where $Leb^{u}$ is the Lebesgue measure induced on the local unstable leaf.  It follows from the definition of the Hausdorff dimension that $\dim_{H}B=u.$ By the definition of the Hausdorff dimension of measures, we get $\dim_H\mu_x^u=u$. Then by (\ref{adofhdm}) one has $\lim_{r\to0}\frac{\log\mu_x^u(B^u(x,r))}{\log r}=u$ for $\mu$ almost every $x$.
\end{proof}

We also denote $\lim_{r\to0}\frac{\log\mu_x^u(B^u(x,r))}{\log r}$ by $\dim_H^u\mu$ for $\mu$ almost every $x$.
Barreira, Pesin and Schmeling \cite{BPS} proved $\dim_H\mu=\dim_H^u\mu+\dim_H^s\mu$. Combining with (\ref{limofHD}), (\ref{star10}) and Lemma \ref{u} one has
\begin{equation*}
\lim_{\varepsilon\to0}\dim_H\Lambda_\varepsilon=\dim_H\mu.
\end{equation*}
This completes the proof of Theorem \ref{Main B}.

\end{document}